\documentclass[twoside]{article}

\usepackage{color}
\usepackage{amssymb}
\usepackage{dsfont}
\usepackage{amssymb,amsmath,ulem,cancel}

\numberwithin{equation}{section}

\newtheorem{theorem}{Theorem}[section]
\newtheorem{lem}{Lemma}[section]
\newtheorem{pro}{Proposition}[section]
\newtheorem{cor}{Corollary}[section]
\newtheorem{rem}{Remark}[section]
\newtheorem{rems}{Remarks}[section]
\newtheorem{ex}{Example}[section]
\newtheorem{defi}{Definition}[section]
\newtheorem{hyp}{Assumption}[section]
\newtheorem{con}{Conjecture}[section]

\newcommand{\bt}{\begin{theorem}}
\newcommand{\et}{\end{theorem}}
\newcommand{\bl}{\begin{lem}}
\newcommand{\el}{\end{lem}}
\newcommand{\bp}{\begin{pro}}
\newcommand{\ep}{\end{pro}}
\newcommand{\bcor}{\begin{cor}}
\newcommand{\ecor}{\end{cor}}
\newcommand{\bcon }{\begin{con} \rm }
\newcommand{\econ }{\end{con}}
\newcommand{\lab }{\label }
\newcommand{\bd}{\begin{defi} \rm }
\newcommand{\ed}{\end{defi}}
\newcommand{\brem }{\begin{rem} \rm }
\newcommand{\erem }{\end{rem}}
\newcommand{\brems }{\begin{rems} \rm }
\newcommand{\erems }{\end{rems}}
\newcommand{\bhyp }{\begin{hyp} \rm }
\newcommand{\ehyp }{\end{hyp}}
\newcommand{\bex}{\begin{ex} \rm }
\newcommand{\eex}{\end{ex}}
\newcommand{\be}{\begin{equation}}
\newcommand{\ee}{\end{equation}}
\newcommand{\bde}{\begin{displaymath}}
\newcommand{\ede}{\end{displaymath}}
\newcommand{\beq}{\begin{eqnarray*}}
\newcommand{\eeq}{\end{eqnarray*}}
\newcommand{\beqa}{\begin{eqnarray}}
\newcommand{\eeqa}{\end{eqnarray}}
\newcommand{\bea}{\begin{align*}}
\newcommand{\eea}{\end{align*}}

\def\proof{\noindent {\it Proof. $\, $}}
\def\endproof{\hfill $\Box$ \vskip 5 pt}

\newcommand{\I}{\mathds{1}}
\newcommand{\Iast}{\mathds{1}^{\!\ast}}
\newcommand{\wh}{\widehat}
\newcommand{\wt}{\widetilde}

\def\phi{\varphi }

\newcommand{\norm}{|\!\!|\!|\!\!|}

\newcommand{\aaee}{\mbox{\rm a.e.}}

\newcommand{\Leb }{\ell }

\newcommand{\rr}{{\mathbb R}}
\newcommand{\ff}{{\mathbb F}}

\def\gg{{\mathbb G}}
\def\hdd{h_d}

\newcommand{\cG}{{\cal G}}

\def\P{\mathbb P}
\def\PT{\wt{\mathbb P}}
\newcommand{\XX}{\mathbb X}

\newcommand{\zzti}{{\widetilde z}^i_t}

\newcommand{\EP}{{\mathbb E}_{\mathbb P}}

\newcommand{\Hlamo}{\wh{\mathcal{H}}_{\lambda}^{2}}
\newcommand{\wHlamo}{\wh{\mathcal{H}}_{\lambda}^{2}}
\newcommand{\wHlamd}{\wh{\mathcal{H}}_{\lambda}^{2,d}}
\newcommand{\wHzero}{\wh{\mathcal{H}}_{0}^{2}}
\newcommand{\wHzerd}{\wh{\mathcal{H}}_{0}^{2,d}}

\newcommand{\sumik}{\textstyle{\sum}}

\newcommand{\Keywords}[1]{\par\noindent{\small{\bf Keywords\/}: #1}}
\newcommand{\Class}[1]{\par\noindent{\small{\bf Mathematics Subjects Classification (2010)\/}: #1}}

\title{{\Large \bf BSDEs DRIVEN BY A MULTI-DIMENSIONAL MARTINGALE AND THEIR
APPLICATIONS TO MARKET MODELS WITH FUNDING COSTS} \vskip 65 pt }

\author{Tianyang Nie and Marek Rutkowski\footnote{The research of Tianyang Nie and Marek Rutkowski
was supported under Australian Research Council's Discovery Projects funding scheme (DP120100895).}
\\ School of Mathematics and Statistics \\ University of Sydney
\\ Sydney, NSW 2006, Australia}

\date{\vskip 50 pt 1 December 2014 \vskip 50 pt}

\begin{document}

\maketitle

\begin{abstract}
We establish some well-posedness and comparison results for BSDEs driven by one- and multi-dimensional martingales.
On the one hand, our approach is largely motivated by results and methods developed in Carbone et al. \cite{CFS-2008} and
El Karoui and Huang \cite{ELH-1997}. On the other hand, our results are also motivated by the recent developments in arbitrage pricing theory under funding costs and collateralization. A new version of the comparison theorem for BSDEs driven
by a multi-dimensional martingale is established and applied to the pricing and hedging BSDEs studied in Bielecki and Rutkowski \cite{BR-2014} and Nie and Rutkowski \cite{NR-2014}. This allows us to obtain the existence and uniqueness results for unilateral prices and to demonstrate the existence of no-arbitrage bounds for a collateralized contract when both agents have non-negative
initial endowments.
\vskip 20 pt
\Keywords{BSDE, comparison theorem, arbitrage pricing, funding costs}
\vskip 20 pt
\Class{60H10,$\,$91G40}
\end{abstract}


\newpage

\section{Introduction} \label{sect1}

The origin of the theory of backward stochastic differential equations (BSDEs) can be traced back to the work by Bismut \cite{B-1973} and, especially, the paper by Pardoux and Peng \cite{PP-1990}, who were the first to consider the general BSDEs driven by a Brownian motion. Since then, the theory of BSDEs attracted a great interest because of its application in stochastic control theory, PDEs and mathematical finance (see, e.g., \cite{EPQ-1997,EQ-1995,HIM-2005,PP-1992,P-1991,RE-2000}).
However, despite the fact that prices of financial assets are usually modeled as semimartingales,
applications of BSDEs in finance beyond the Brownian setting are relatively rare (see, e.g., \cite{CCR-2014,CCR-2014b,MS-2005,MT-2003c,M-2009}).  The BSDEs driven by a semimartingale were already introduced by
Chitashvili \cite{C-1983}, but since then BSDEs driven by a general martingale were not extensively studied.
Carbone et al. \cite{CFS-2008} and El Karoui and Huang \cite{ELH-1997} examined BSDEs driven by a  c\`{a}dl\`{a}g
martingale without postulating the predictable representation property (PRP), whereas  Li \cite{L-2003} examined
BSDEs driven by a one-dimensional continuous martingale enjoying the PRP.

In this note, we first establish some results for BSDEs driven by one- and multi-dimensional martingales.
Our approach is largely motivated by results and methods developed in Carbone et al. \cite{CFS-2008} and El Karoui and Huang \cite{ELH-1997}. However, for simplicity of presentation, we only consider here the BSDEs driven by continuous martingales with the PRP, whereas in \cite{CFS-2008,ELH-1997} the authors studied the BSDEs driven by a c\`{a}dl\`{a}g martingale $M$ without postulating the PRP, but under the additional assumption that the underlying filtration is quasi-left continuous. Under the latter assumption, the predictable quadratic variation $\langle M\rangle$ of $M$ is continuous. It is worth noting that all results and a priori estimates established in this note will still be valid in this more general framework.

Our main goal is to prove the existence, uniqueness and comparison theorems covering the BSDEs introduced \cite{BR-2014,NR-2014,NR-2014a}, where the pricing and hedging of contingent claims in financial models with funding costs is studied. We mention that Mocha and Westray \cite{MW-2012} and Tevzadze \cite{T-2008}, established a comparison theorem for the case of the quadratic growth and, in \cite{T-2008}, for a special choice of a generator. For the linear growth case, Carbone et al. \cite{CFS-2008} gave the comparison results for BSDEs driven by a one-dimensional c\`{a}dl\`{a}g martingale  (see Theorem 2.2 in \cite{CFS-2008}). To this end, they used the Dol\'{e}ans exponential of a c\`{a}dl\`{a}g martingale and they imposed the requirement it is a positive, uniformly integrable martingale and, in addition, satisfies some integrability conditions (see, in particular, Lemma 2.2 in \cite{CFS-2008} or equation \eqref{condis} in Section \ref{sect2} of this work), which are not easy to verify and may be too restrictive for  applications). We stress in this regard that  the results from \cite{CFS-2008,MW-2012,T-2008} are not sufficient for the purposes studied in \cite{BR-2014,NR-2014,NR-2014a}, since the assumptions made in these papers fail to
hold in the context of a typical financial model. Consequently, some extensions of the existing comparison theorems for BSDEs driven by multi-dimensional martingales are needed to demonstrate the existence of non-empty intervals for fair bilateral prices (or bilaterally profitable prices), as well as the monotonicity of prices with respect to the initial endowment of an agent. In Section \ref{sect5}, we show that Theorems \ref{comparison theorem 2x} and \ref{appendix wellposedness theorem for lending special BSDE} are suitable tools to handle BSDEs derived in market model with funding costs when dealing with a collateralized contract.
For further applications of results from this work, we refer to Section 5 in \cite{NR-2014} and Section 3.3 in  \cite{NR-2014a}.
To summarize, our main goal is to extend results from \cite{CFS-2008} for BSDEs driven by multi-dimensional martingales and to relax rather stringent assumptions postulated in \cite{CFS-2008}.

This work is organized as follows. In Section  \ref{sect2}, we recall some definitions and results from \cite{CFS-2008}
and we consider the extended BSDEs driven by one-dimensional continuous martingale. In Section  \ref{sect3}, we study BSDEs driven by a multi-dimensional continuous martingale $M$. We first obtain the existence, uniqueness, and stability results for solutions to these BSDEs under the assumption that the generator satisfies to $m$-Lipschitz condition. Next, we prove the comparison theorem
(see Theorem \ref{comparison theorem 2x}) for BSDEs with a uniformly $m$-Lipschitzian generator. The goal of Section \ref{sect4} is to analyze alternative assumptions regarding the process $m$ arising in representation  \eqref{edded} of the quadratic variation $\langle M\rangle$ (see Assumptions \ref{assumption 2 for m}  and \ref{assumption 3 for m}). We conclude the paper by demonstrating in Section \ref{sect5} that the assumptions of Theorems \ref{comparison theorem 2x} and \ref{appendix wellposedness theorem for lending special BSDE} are satisfied by a particular class of BSDEs that arise in the context of financial models with funding costs studied in related papers \cite{BR-2014,NR-2014,NR-2014a}. We also show that our comparison result is a suitable tool for deriving the bounds for unilateral prices of a collateralized contract in a financial market with funding costs.

\newpage

\section{BSDEs Driven by a One-Dimensional Martingale} \label{sect2}

Let us stress that main goal of this work is to examine BSDEs driven by a multi-dimensional martingale.
For the sake of completeness, we first provide in this section a minor extension of results from Carbone et al. \cite{CFS-2008} when a one-dimensional driving martingale $M$ is complemented by a predetermined driving process $U$ (see equation \eqref{general BSDE 2}), which arises in financial applications studied in Section \ref{sect5}. Moreover, we provide here a discussion of assumptions made in  \cite{CFS-2008}, since we aim to relax some of them in the foregoing section.

We assume that we are given a filtered probability space $(\Omega, {\cal G}, \gg , \P)$ satisfying the usual conditions of right-continuity and completeness. As customary, we assume that the $\sigma$-field ${\cal G}_0$ is trivial.
Let $M$ be a real-valued, continuous, square-integrable martingale on this space.
We postulate that $M$ has the predictable representation property with respect to the filtration $\gg$ under $\P$.
We denote by $\langle M\rangle$ the quadratic variation process of $M$, which is a continuous, increasing, $\gg$-adapted
process vanishing at zero such that $M^{2}-\langle M\rangle$ is a continuous (uniformly integrable) martingale.
We introduce the following notation, for any non-negative constant $\lambda $, \hfill \break
$\Hlamo$ -- the subspace of all real-valued,  $\gg$-adapted processes $X$ satisfying
\be \label{eddedvxx}
\|X\|_{\Hlamo}^{2}:= \EP \bigg[ \int_{0}^{T}e^{\lambda \langle M\rangle_{t}} X^2_{t}\,
d\langle M\rangle_{t} \bigg]<\infty ,
\ee
$L^{2}_{\lambda}$ -- the space of all real-valued, $\mathcal{G}_{T}$-measurable random variables $\eta$ such that
\bde
\|\eta\|_{L^2_{\lambda}}^{2} := \EP \Big[e^{\lambda \langle M\rangle_{T}} \eta^{2} \Big]<\infty.
\ede
Note that $L^{2}_{0} = L^2 (\rr)$ is the space of  $\mathcal{G}_{T}$-measurable,
square-integrable random variables.

\subsection{BSDEs with a Uniformly Lipschitzian Generator} \label{sect2.1}

Assume that we are given a real-valued, continuous martingale $M$, a process $U$, and a random variable $\eta $.
We consider the following BSDE driven by $M$ and $U$, with generator $h$ and the terminal value $\eta $, for $t \in [0,T]$,
\be \label{general BSDE 2}
\left\{
\begin{array}
[c]{ll}
dY_t  = Z_t \, dM_t - h(t, Y_t , Z_t )\, d\langle M\rangle_{t} + dU_t ,\medskip\\
Y_{T}= \eta .
\end{array}
\right.
\ee

In Section  \ref{sect2}, we work under the following standing assumption imposed on the generator $h$.

\bhyp \label{assumption 1 for the driver of the general BSDE}
The generator $h:\Omega \times[0,T] \times \rr \times \rr \rightarrow \rr$ is a ${\cal G}\otimes\mathcal{B}([0,T])\otimes\mathcal{B}(\rr)\otimes\mathcal{B}(\rr)$-measurable function such that $h(\cdot,y,z)$ is a $\gg$-adapted process for any fixed $(y,z)\in\rr \times \rr$ and the process $h(\cdot,0,0)$ belongs to $\Hlamo $.
\ehyp

We adopt the following definition of a solution to BSDE (\ref{general BSDE 2}). It is clear from this definition
that we restrict our attention to solutions $(Y,Z)$ from the space $\Hlamo \times \Hlamo$.

\bd \lab{defrr}
A {\it solution} to BSDE (\ref{general BSDE 2}) is a pair $(Y,Z)\in \Hlamo \times \Hlamo$ of
processes satisfying (\ref{general BSDE 2}), $\mathbb{P}$-a.s., such that $Z$ is $\gg$-predictable and
$Y-U$ is a continuous process.
\ed

We emphasize that Definition \ref{defrr} postulates the continuity of the process $Y-U$, rather than $Y$. Obviously, if
we assume that $U$ is $\gg$-progressively measurable (resp. $\gg$-predictable or continuous), then $Y$ will share this property as well.  For any natural $d$, we denote by $\| \cdot \|$ the Euclidean norm in ${\mathbb R}^d$.
In this section, we set $d=1$, but the next definition also applies to the multi-dimensional case.

\bd \label{ulip}
We say that the generator $h$ satisfies the {\it uniform Lipschitz condition} if there exists a constant $L$ such that, for all $t\in[0,T]$ and all $y_{1},y_{2}, z_{1},z_{2}\in\rr^d$,
\be  \label{uniformly Lipschitz for the driver}
|h(t,y_{1},z_{1})-h(t,y_{2},z_{2})|\leq L\left(|y_{1}-y_{2}|+\|z_{1}-z_{2}\|\right). 
\ee
\ed

The next definition hinges on a minor adjustment of the terminology used in Carbone et al. \cite{CFS-2008} and El Karoui et al.  \cite{EPQ-1997}

\bd \lab{stand}
We say that $(h, \eta ,U)$ is a $(\lambda ,L)$-{\it standard parameter} if: \hfill \break
(i) $h:\Omega \times[0,T] \times \rr \times \rr \rightarrow \rr$ satisfies the uniform Lipschitz
condition \eqref{uniformly Lipschitz for the driver} with a constant $L$, \hfill \break
(ii) the process $h(\cdot,0,0)$ belongs to $ \Hlamo $, \hfill \break
(iii) a random variable $\eta $ belongs to $L^{2}_{\lambda}$,  \hfill \break
(iv) a real-valued, $\gg$-adapted process $U$ belongs to $\Hlamo $ and $U_{T}\in L^{2}_{\lambda}$.
\ed

The next result is an almost immediate consequence of Theorem 2.1 in \cite{CFS-2008}, where the case
of $U=0$ was examined.

\begin{theorem}
Assume that $(h,\eta ,U )$ is a $(\lambda ,L)$-standard parameter for some $\lambda>\lambda_{0}(L)$ where
\bde
\lambda_{0}(L):=\left\{
\begin{array}
[c]{ll}
2\sqrt{2}L,  & L\leq\frac{\sqrt{2}}{2},\medskip\\
2L^{2}+1,    & L>\frac{\sqrt{2}}{2}.
\end{array}
\right.
\ede
Then BSDE (\ref{general BSDE 2}) has a unique solution $(Y,Z)\in \Hlamo \times \Hlamo $. Moreover, the process $Y-U$ satisfies
\bde
\EP \bigg[\sup_{t\in[0,T]}e^{\lambda\langle M\rangle_{t}} (Y_{t}-U_{t})^{2}\bigg]<\infty.
\ede
\end{theorem}

\begin{proof}
Let us set $\wh Y_t = Y_t - U_t$. Then
\be \label{transferred general BSDE 1}
\wh Y_t  = \wh \eta - \int_t^T Z_u \, dM_u  +\int_t^T \wh h(u, \wh Y_u , Z_u )\, d\langle M\rangle_{u}
\ee
where $\wh \eta = \eta - U_T$ and $\wh h(t,\wh Y_t, Z_t ) : =h(t,\wh Y_t + U _t , Z_t )$.
Since, by assumption, the processes $h(\cdot,0,0)$ and $U$ belong to $\Hlamo $ and condition
(\ref{uniformly  Lipschitz for the driver}) holds, we have $\wh h(\cdot,0, 0)=h(\cdot,U , 0)\in \Hlamo$.
Moreover, it is easy to check that $\wh \eta \in L^{2}_{\lambda}$ and $\wh h$ satisfies (\ref{uniformly Lipschitz for the driver}). Therefore, $(\wh{h}, \wh \eta )$ is also a $(\lambda ,L)$-standard parameter and thus BSDE (\ref{transferred general BSDE 1}) has a unique solution $(\wh{Y},Z)\in \Hlamo \times \Hlamo $, by virtue of Theorem 2.1 in \cite{CFS-2008}. Moreover, the process $\wh{Y}$ is continuous and it satisfies
\bde
\EP \bigg[\sup_{t\in[0,T]}e^{\lambda\langle M\rangle_{t}}\wh{Y}^2_{t}\bigg]<\infty.
\ede
We conclude that the pair $(Y,Z)$ with $Y:=\wh Y + U \in \Hlamo$ is a unique solution to BSDE (\ref{general BSDE 2}).
\end{proof}

\subsection{Comparison Theorem: One-Dimensional Case}      \label{2.2}

We now focus on the comparison theorem of BSDE driven by a one-dimensional continuous martingale.
Let $\mathcal{E}(M)$ denote the Dol\'eans exponential of a continuous martingale $M$, that is,
\bde
\mathcal{E}_t(M):=\exp\left\{M_{t}-\frac{1}{2}\langle M\rangle_{t}\right\}.
\ede
From Novikov's criterion (see, for instance, Corollary 1.1 in Kazamaki \cite{K-1994}), it is known that if $\langle M\rangle_{T}$ is bounded, then $\EP [ \mathcal{E}_t(M)]=1$ for all $t \in [0,T]$.

For a given function $h:\Omega \times[0,T] \times \rr \times \rr \rightarrow \rr$, we introduce the following processes
\begin{align*}
&\Delta_{Y}h_{t}=\frac{h(t,Y^{1}_{t},Z^{1}_{t})-h(t,Y^{2}_{t},Z^{1}_{t})}{Y^{1}_{t}-Y^{2}_t}\,\I_{\{Y^{1}_{t}\neq Y^{2}_t\}}, \\
&\Delta_{Z}h_{t}=\frac{h(t,Y^{2}_{t},Z^{1}_{t})-h(t,Y^{2}_{t},Z^{2}_{t})}{Z^{1}_{t}-Z^{2}_t}\, \I_{\{Z^{1}_{t}\neq Z^{2}_t\}}.
\end{align*}
From Theorem 2.2 in Carbone et al. \cite{CFS-2008} (see also conditions (i)--(ii) in Lemma 2.2 in  \cite{CFS-2008}),
we obtain the following version of the comparison theorem for solutions to BSDEs.

\begin{theorem} \label{comparison theorem 1}
Let $(h^i, \eta^i ,U)$ be a $(\lambda^i ,L^i)$-standard parameter and
let $(Y^{i},Z^{i})$ be the unique solution of the following BSDE, for $i=1,2$,
\be\label{general BSDE 3}
\left\{
\begin{array}
[c]{ll}
dY_t^{i} =Z^{i}_t \, dM_t - h^{i}(t, Y^{i}_t , Z^{i}_t )\,d\langle M\rangle_{t} + dU_t,\medskip\\
Y_{T}^{i}=\eta^{i}.
\end{array}
\right.
\ee
Suppose that $\mathcal{E}\big( \int_{0}^{\cdot }\Delta_{Z}h^{1}_{u}\, dM_{u}\big)$ is a positive,
uniformly integrable martingale and
\be \label{condis}
\EP \left[\left(\sup_{t\in[0,T]}\exp\left\{\int_{0}^{t}\Delta_{Y}h^{1}_{u}\,d\langle M\rangle_{u}\right\}\right)^{2}
\left(\mathcal{E}_T \Big(\int_{0}^{\cdot }\Delta_{Z}h^{1}_{u}\, dM_{u}\Big)\right)^{2}\right]<\infty.
\ee
If $\eta^{1}\ge\eta^{2}$ and $h^{1}(\cdot ,Y^{2} , Z^{2})\ge h^{2}(\cdot ,Y^{2}, Z^{2}),\, \P\otimes \Leb$-a.e., then
$Y^{1}_{t}\ge Y^{2}_{t}$ for any $t\in[0,T]$.
\end{theorem}

Alternatively, one can consider the following assumptions: $h^{1}(\cdot ,Y^{1}, Z^{1})\ge h^{2}(\cdot ,Y^{1}, Z^{1})$,  $\P\otimes \Leb$-a.e., $\eta^{1}\ge\eta^{2}$, $\mathcal{E}(\int_{0}^{\cdot }\Delta_{Z}h^{2}_{u}\,dM_{u})$
is a positive uniformly integrable martingale, and
\bde
\EP \left[\left(\sup_{t\in[0,T]}\exp\left\{\int_{0}^{t}\Delta_{Y}h^{2}_{u}\,d\langle M\rangle_{u}\right\}\right)^{2}
\left(\mathcal{E}_T \Big(\int_{0}^{\cdot }\Delta_{Z}h^{2}_{u}\, dM_{u}\Big)\right)^{2}\right]<\infty.
\ede
Then the assertion of Theorem \ref{comparison theorem 1} is still valid.

Let us make some important observations regarding condition \eqref{condis} in Theorem \ref{comparison theorem 1}.
Since $h^1$ satisfies the uniform Lipschitz condition, the processes $|\Delta_{Y}h^1_{u}|$ and $|\Delta_{Z}h^1_{u}|$
are both bounded by $L$. Next, if we assume that $M$ is a continuous, square-integrable martingale then, from Remark 1.3 in \cite{K-1994},  we deduce that the Dol\'eans exponential $\mathcal{E}(\int_{0}^{\cdot }\Delta_{Z}h^1_{u}\,dM_{u})$
is a positive, uniformly integrable martingale. Furthermore, if we assume that $\langle M\rangle_{T}$ is bounded, then $\langle \int_{0}^{\cdot }\Delta_{Z}h^1_{u}\,dM_{u}\rangle_T$ is also bounded and thus, using Novikov's criterion,
we conclude condition \eqref{multid BSDE} is satisfied. It is worth stressing that the postulate that
$\langle M\rangle_{T}$ is bounded is very restrictive, since it is not likely to hold in further applications
of BSDEs driven by a martingale.

\brem \label{remark for cadlag martingale}
Let us now consider the case where $M$ is a c\`{a}dl\`{a}g martingale. Then the Dol\'{e}ans exponential $\mathcal{E}(M)$
is the unique solution of the SDE
\bde
\mathcal{E}_{t}(M)=1+\int_{0}^{t} \mathcal{E}_{u-}(M)\, dM_{u}.
\ede
It is known that  $\mathcal{E}(M)$ is a local martingale and equals (we denote $\Delta M_{u}=M_{u}-M_{u-}$)
\bde
\mathcal{E}_{t}(M)=\exp \left\{M_{t}-M_{0}-\frac{1}{2}\langle M \rangle_{t}\right\}\prod_{0< u \leq t}(1+\Delta M_{u})e^{-\Delta M_{u}}
\ede
where $\langle M \rangle = \langle M^c \rangle$ (as usual, $M^c$ stands for the continuous martingale part of $M$).
If the martingale $M$ is square-integrable with $\Delta M_{s}>-1$ and $\langle M\rangle_{T}$ is bounded, then $\mathcal{E}(M)$ is a positive, square-integrable (thus uniformly integrable) martingale (see \cite{LM-1978} or Remark 1.3 in \cite{K-1994}). Therefore, if $\Delta \big( \int_{0}^{\cdot }\Delta_{Z}h^{1}_{u}\, dM_{u} \big) >-1$  and $\langle \int_{0}^{\cdot}\Delta_{Z}h^{1}_{u}\, dM_{u}\rangle_{T}$ is bounded, we still have that $\mathcal{E}\big(\int_{0}^{\cdot }\Delta_{Z}h^{1}_{u}\, dM_{u}\big)$ is a positive, uniformly integrable martingale. Hence if $\langle M\rangle_{T}$ is bounded, then condition \eqref{condis} holds.
\erem

\brem \label{remark for multi-dimensional case}
To directly extend the results of this subsection to BSDEs driven by a multi-dimensional martingale $M=(M^1,\dots ,M^d)^*$,
one could consider the following generalization of BSDE  \eqref{general BSDE 2}
\be \label{multid BSDE}
dY_t  = Z_t^{\ast} \, dM_t -  \Iast d\langle M\rangle_{t}\, \hdd(t, Y_t , Z_t )\, + dU_t ,\quad Y_{T}= \eta
\ee
with the $\rr^d$-valued generator $\hdd: \Omega \times[0,T] \times \rr \times \rr^d \rightarrow \rr^d$ where
$\Iast =(1,1, \dots ,1)$.
Unfortunately, the term $ \Iast d\langle M\rangle_{t}\, \hdd(t, Y_t , Z_t )$ in \eqref{multid BSDE} seems to be rather untractable and thus we focus on BSDE \eqref{special BSDE} driven by a multi-dimensional martingale in which the matrix-valued process $\langle M\rangle$ can be factorized (see equation \eqref{mm33}), although we also make some comments about solvability of BSDE \eqref{multid BSDE} in Section \ref{sect4.1}.

Suppose that we manage to prove the comparison theorem for BSDE (\ref{multid BSDE}) driven by a multi-dimensional martingale analogous to Theorem~\ref{comparison theorem 1}. Since the boundedness of $\langle M\rangle_{T}$ may fail to hold, typically,
a straightforward application of a multi-dimensional extension of Theorem \ref{comparison theorem 1} would not be possible anyway, since it would require to verify the conditions imposed on the Dol\'eans exponential $\mathcal{E}(M)$ and this task is rather hard.

In next section, we will study BSDEs driven by multi-dimensional continuous martingales, since such BSDEs play an important r\^ole in numerous financial applications where market models with several risky assets are introduced and studied.  Our main goal is to establish a version of a comparison theorem in which, in particular, the boundedness of $\langle M\rangle_{T}$ is not postulated (see Theorem~\ref{comparison theorem 2x}).
\erem

\section{BSDEs Driven by a Multi-Dimensional Martingale} \lab{sect3}

In this section, we first revisit results from \cite{CFS-2008,ELH-1997} and we establish in Section \ref{sect3.1} their extensions to the case when an additional driving term $U$ appears in our BSDE under the $m$-Lipschitz condition  \eqref{Lipschitz for the driver}. In Section \ref{sect3.2}, we study the special case when the generator satisfies the uniform
$m$-Lipschitz condition  \eqref{Lipschitz for the driverx}.  The goal of Section \ref{sect3.3} is to establish
the main comparison theorem for BSDEs driven by a multi-dimensional martingale (see Theorem \ref{comparison theorem 2x}).

Let $M=(M^{1},M^{2},\ldots,M^{d})^{\ast}$ (by $^{\ast}$, we denote the transposition), where the processes $M^i,\, i=1,2,\dots , d$ are continuous, square-integrable martingales on the filtered probability space $(\Omega, {\cal G}, \gg , \P)$. We postulate
that $M$ has the predictable representation property with respect to the filtration $\gg$ under $\P$. We denote
by $\langle M \rangle$ the quadratic (cross-) variation process of $M$, so that $\langle M \rangle_{t}$ takes values in
$\rr^{d \times d}$ and the $(i,j)$th entry of the matrix $\langle M \rangle_{t}$ is $\langle M^{i},M^{j}\rangle_{t}$.
As in \cite{CFS-2008,ELH-1997}, we henceforth work under the following standing assumption regarding the continuous
process of finite variation $\langle M \rangle$, so that it is implicitly assumed that the processes $m$ and $Q$
in equation \eqref{mm33} are given.

\bhyp  \label{edded}
We assume that there exists an $\rr^{d \times d}$-valued $\gg$-adapted process $m$ and a $\gg$-adapted, continuous,
bounded, increasing process $Q$ with $Q_0=0$ such that, for all $t \in [0,T]$,
\be \label{mm33}
\langle M\rangle_{t}=\int_{0}^{t} m_{u} m_{u}^{\ast}\, dQ_{u}.
\ee
\ehyp

From Proposition 2.9 in Chapter II of Jacod and Shiryaev \cite{JS-2003} (see also \cite{ELH-1997,MW-2012,M-2009}), we know that Assumption  \ref{edded} is met by an arbitrary continuous, square-integrable martingale and the factorization \eqref{mm33} of the process $\langle M\rangle$ is not unique. For instance, if we set $Q:=\arctan(\sum_{i=1}^{d}\langle M^{i},M^{i}\rangle)$, then $Q$ is $\gg$-adapted, continuous, increasing process, which is bounded by $\frac{\pi}{2}$. Moreover, the Kunita-Watanabe inequality shows that for all $1\leq i,j\leq d$ the process $\langle M^{i}, M^{j}\rangle$ is absolute continuous with respect to $Q$, and thus the Radon-Nikodym theorem allows us to obtain an $\rr^{d \times d}$-valued, $\gg$-predictable process $D$, which is positive semi-definite. Furthermore, we can factorize $D$ as $D=mm^{\ast}$ for an $\rr^{d \times d}$-valued, $\gg$-predictable processes $m$. In particular, if $M$ is a $d$-dimensional Brownian motion, then we can simply choose $Q_{t}=t$ for all $t\in[0,T]$ and $m=I$, where $I$ stands for the $d$-dimensional identity matrix.

Let us make some comments on alternative technical assumptions regarding the measurability of $m$. In some papers, such as \cite{ELH-1997,MW-2012,M-2009}, the authors take $m$ as a predictable process, which can be constructed as above, for instance.  However, on the one hand, usually it is sufficient to take $m$ as an adapted process to obtain the well-posedness of BSDEs. On the other hand, when we consider the stochastic integral with respect to $M$, where $m$ appears in the integrand,  usually it suffices that $m$ is progressively measurable. In particular, if $Q_{t}=t$ (which, obviously, implies that the process $\langle M\rangle $ is absolutely continuous with respect to the Lebesgue measure), then it is enough to postulated that $m$ is adapted (for more details, see Remark 2.11 in Chapter 3 of  Karatzas and Shreve \cite{KS-1998}).

It is also clear that $m_{u}m_{u}^{\ast}$ is a square matrix and it is positive semi-definite, so that $(m_{u}m_{u}^{\ast})^{\frac{1}{2}}$ is well-defined (for notation, see Remark \ref{notat55}). If $m_{u}m_{u}^{\ast}$ is positive definite, then $m_{u}m_{u}^{\ast}$ is invertible, and thus we can also define $(m_{u}m_{u}^{\ast})^{-\frac{1}{2}}$. Finally,
 we note that if $m_u$ is a symmetric matrix (i.e., $m_u=m_u^{\ast}$), then $m_{u}=m^{\ast}_{u}=(m_{u}m_{u}^{\ast})^{\frac{1}{2}}$.

Regarding the symmetry of $m$ in \eqref{mm33}, observe that the condition $\langle M\rangle_{t}=\int_{0}^{t}m_{u}m_{u}^{\ast}\,dQ_{u}$
with $m$ being an $\rr^{d \times d}$-valued, $\gg$-adapted process is equivalent to
$\langle M\rangle_{t}=\int_{0}^{t}\bar{m}_{u}\bar{m}_{u}^{\ast}\,dQ_{u}$ with $\bar{m}$ being a symmetric
$\rr^{d \times d}$-valued, $\gg$-adapted process. Indeed, it suffices to take $\bar{m}_{u}=(m_{u}m_{u}^{\ast})^{\frac{1}{2}}$.
Therefore, without loss of generality, we may and do assume that $m$ takes values in the space of symmetric matrices.

\brem \label{notat55}
We denote by $\norm a \norm $ the norm of a $d \times d$ matrix $a$, where $\norm a \norm^2 :=\text{Tr}(a a^{\ast})$. For a
positive semi-definite matrix $a$, we denote by $a^{\frac{1}{2}}$ the unique square root of $a$, i.e., $a^{\frac{1}{2}}a^{\frac{1}{2}}=a$. Recall that there exists an orthogonal matrix $O$ and a diagonal matrix $b$ with non-negative diagonal elements such that $a=O^{\ast}b O$. The square root of $b$ is also a diagonal matrix, denoted by $b^{\frac{1}{2}}$, with diagonal elements equal to square roots of diagonal elements of $b$. Then we set $a^{\frac{1}{2}}=O^{\ast}b^{\frac{1}{2}}O$. Moreover, if $a$ is positive definite, then the inverse of $a^{\frac{1}{2}}$, denoted as $a^{-\frac{1}{2}}$, is well defined.
\erem

\subsection{BSDEs with an $m$-Lipschitzian Generator}   \label{sect3.1}

In this section, we study the following BSDE driven by a $d$-dimensional martingale $M = (M^1,\dots , M^d)^*$
and a real-valued, $\gg$-adapted process $U$, for $t \in [0,T]$,
\be \label{special BSDE}
\left\{
\begin{array}
[c]{ll}
dY_t  = Z_t^{\ast} \, dM_t -h(t, Y_t , Z_t )\,dQ_{t} + dU_t,\medskip\\
Y_{T}=\eta .
\end{array}
\right.
\ee
Recall that we work under the standing assumption that $M$ satisfies Assumption \ref{edded}. We also make throughout
the following technical assumption regarding the measurability of the generator $h$.

\bhyp  \label{assumption standard generator}
Let $h:\Omega \times[0,T] \times \rr \times \rr^{d} \rightarrow \rr$ be a ${\cal G}\otimes\mathcal{B}([0,T])\otimes\mathcal{B}(\rr)\otimes\mathcal{B}(\rr^d)$-measurable function such that $h(\cdot,\cdot,y,z)$ is a $\gg$-adapted process for any fixed $(y,z)\in\rr \times \rr^d$.
\ehyp

Before stating the definition of a solution
to BSDE \eqref{special BSDE}, we need to introduce some notation and define suitable spaces of processes in which will search
for solutions. To this end, following \cite{CFS-2008}, we first introduce the following version of the
Lipschitz condition for the generator $h$.

\bd \label{mlip}
We say that $h$ satisfies the {\it $m$-Lipschitz condition} if there exist two strictly positive and $\gg$-adapted processes $\rho $ and $\theta$ such that, for all $t\in[0,T]$ and $y_{1},y_{2}\in\rr,\, z_{1},z_{2}\in\rr^d$,
\be \label{Lipschitz for the driver}
|h(t,y_{1},z_{1})-h(t,y_{2},z_{2})|\leq \rho_{t}|y_{1}-y_{2}|+\theta_{t}\| m_{t}^{\ast}(z_{1}-z_{2})\|.
\ee
\ed

We set $\alpha_{t}^{2}=\rho_{t}+\theta_{t}^{2}$ for $t \in [0,T]$ and we define
the process $N_{t}:=\int_{0}^{t} \alpha_{u}^{2} \,dQ_{u}$ for $t \in [0,T]$.
For a fixed $\lambda\ge0$, we denote by $\wHlamd$ the subspace of all $\rr^{d}$-valued, $\gg$-adapted processes $X$ with the norm  $\|\cdot \|_{\wHlamd }$ given by
\be \label{defhh}
\|X\|_{\wHlamd}^{2}:=\EP \bigg[ \int_{0}^{T}e^{\lambda N_{t}}\|X_{t}\|^{2}\,dQ_{t} \bigg] <\infty .
\ee
Let $\wh{L}^{2}_{\lambda}$ stand for the space of all real-valued, $\mathcal{G}_{T}$-measurable
random variables $\eta$ such that
\be \label{defhhh}
\|\eta \|_{\wh{L}^2_{\lambda}}^{2}=\EP \left[e^{\lambda N_{T}}\eta^{2}\right]<\infty .
\ee
It is clear that the spaces in which we will search for solutions depend on the generator $g$.
In fact, from the next definition, it transpires that the process $m$ is also used for this purpose.

\bd
A {\it solution} to BSDE (\ref{special BSDE}) is a pair $(Y,Z)$ of $\gg$-adapted processes satisfying (\ref{special BSDE}) for all
$t \in [0,T]$ and such that: $(\alpha Y, m^{\ast}Z)\in\wHlamo  \times \wHlamd$, the process $Z$ is $\gg$-predictable, and the process $Y-U$ is continuous.
\ed

Note that if, in addition, the process $U$ is $\gg$-progressively measurable (resp. $\gg$-predictable), then $Y$ is $\gg$-progressively measurable (resp. $\gg$-predictable) as well.

One can object that the definition of a solution to BSDE (\ref{special BSDE}) under the $m$-Lipschitz condition
is somewhat artificial, since it is tailored to the method of the proof of the existence and uniqueness theorem
(see, for instance, Theorem 3.2 in \cite{CFS-2008}). In the next subsection, we will impose a stronger uniform $m$-Lipschitz condition and we will reduce the complexity of the definition of a solution to BSDE (\ref{special BSDE}).

The next definition is a counterpart of Definition \ref{stand} of a $(\lambda ,L)$-standard parameter.
To be more precise, we deal here with the notion of the $(\lambda , m ,Q, \rho , \theta )$-{\it standard parameter}
but, for the sake of conciseness, we decided to call it simply a $(\lambda , m )$-{\it standard parameter}.

\bd \lab{defstan}
We say that the triplet $(h,\eta ,U)$ is a $(\lambda , m )$-{\it standard parameter} if: \hfill \break
(i)  $h:\Omega \times [0,T] \times \rr \times \rr^{d} \rightarrow \rr$ satisfies the $m$-Lipschitz
condition with processes $\rho $ and $\theta $, \hfill \break
(ii) the process $\alpha^{-1}h(\cdot,0,0)$ belongs to $\wHlamo$, \hfill \break
(iii) a random variable $\eta $ belongs to $\wh{L}^{2}_{\lambda}$, \hfill \break
(iv) $U$ is a real-valued, $\gg$-adapted process
such that $\alpha U\in\wHlamo$ and $U_{T}\in \wh{L}^{2}_{\lambda}$.
\ed

The proof of the next result hinges on Theorem 3.2 in Carbone et al. \cite{CFS-2008}.

\begin{theorem} \label{1 wellposedness theorem for special BSDE}
Let $(h,\eta ,U)$ be a $(\lambda , m)$-standard parameter for some $\lambda>3$.
Then BSDE (\ref{special BSDE}) has a unique solution $(Y,Z)$ such that $(\alpha Y, m^{\ast}Z)\in\wHlamo
 \times \wHlamd $. Moreover, the process $Y-U$ satisfies
\bde 
\EP \bigg[\sup_{t\in[0,T]}e^{\lambda N_{t}}(Y_{t}-U_{t})^{2} \bigg]<\infty.
\ede
\end{theorem}

\begin{proof}
We set $\wh Y_t := Y_t - U_t$, so that  (\ref{special BSDE}) becomes
\be \label{transferred special BSDE}
\wh Y_t  = \wh \eta - \int_t^T Z^*_u \, dM_u + \int_t^T \wh h(t, \wh Y_u , Z_u )\, dQ_{u}
\ee
where $\wh \eta := \eta - U_T$ and $\wh h(t,\wh Y_t, Z_t ) : =h(t,\wh Y_t + U _t , Z_t )$. Since $h(\cdot,0,0)\in\wHlamo$, $\alpha U\in\wHlamo$ and (\ref{Lipschitz for the driver}), we have $\wh h(\cdot,0, 0)=h(\cdot,U, 0)\in\wHlamo$. Moreover, it is easy to check that $\wh \eta \in \wh{L}^{2}_{\lambda}$ and the function $\wh h$ satisfies (\ref{Lipschitz for the driver}). Therefore, $(\wh{h}, \wh{\eta} )$ is a $(\lambda , m)$-standard parameter as well. Consequently, from Theorem 3.2 in \cite{CFS-2008}, we deduce that BSDE (\ref{transferred special BSDE})
has a unique solution $(\wh{Y},Z)$ such that $(\alpha \wh{Y}, m^{\ast}Z)\in\wHlamo \times \wHlamd$. Moreover, $\wh{Y}$ is continuous and satisfies $\EP \big[\sup_{t\in[0,T]}e^{\lambda N_{t}}\wh{Y}^2_{t} \big]<\infty$.
It is now not hard to check that $(Y,Z)$ with $Y:=\wh Y + U$ is a unique solution to BSDE (\ref{special BSDE}) and $(\alpha Y, m^{\ast}Z) \in \wHlamo  \times \wHlamd $, since $\alpha \wh{Y} \in\wHlamo$ and $\alpha U\in\wHlamo$.
\end{proof}

We also have the following stability result, which extends Proposition 3.1 in \cite{CFS-2008}.

\bp \label{1 stability of BSDE}
Let $(h^i,\eta^i,U^i)$ be a $( \lambda , m)$-standard parameter with $\lambda>3$ and let $(Y^i,Z^i)$
be the solution to the following BSDE, for $i=1,2$,
\be\label{special BSDE 2}
\left\{
\begin{array}
[c]{ll}
dY_t^{i}  =Z^{i,\ast}_t \, dM_t - h^{i}(t, Y^{i}_t , Z^{i}_t )\,dQ_{t} + dU_t^{i},\medskip\\
Y_{T}^{i}=\eta^{i}.
\end{array}
\right.
\ee
If we denote $Y=Y^{1}-Y^{2}$, $Z=Z^{1}-Z^{2}$, $U=U^{1}-U^{2}$, $\eta=\eta^{1}-\eta^{2}$ and  $h=h^{1}-h^{2}$, then
\bde 
\EP \bigg[\sup_{t\in[0,T]}e^{\lambda N_{t}}(Y_{t}-U_{t})^{2}\bigg]+\|\alpha Y\|^2_{\wHlamo}+\|m^{\ast}Z\|^2_{\wHlamd }
\leq K_{1}\Delta
\ede
where $K_1$ is a constant and
\bde
\Delta=\left[ \|\eta-U_{T}\|^{2}_{\wh{L}^2_{\lambda}}+\|\alpha U\|^2_{\wHlamo }
+ \| \alpha^{-1}h(t,Y^{2}-U^{2},Z^{2})\|^2_{\wHlamo}\right].
\ede
\ep

\begin{proof}
From Theorem \ref{wellposedness theorem for special BSDE}, we know that there exists a unique solution $(Y^{i},Z^{i})$ of BSDEs (\ref{special BSDE 2}) for $i=1,2$. Let $\wh{Y}^{i}=Y^{i}-U^{i}$ and $\wh{Y}=\wh{Y}^{1}-\wh{Y}^{2}$.
Then, similarly as in the proof of Proposition 3.1 in \cite{CFS-2008}, we obtain (note that the value of a constant $K_{1}$ may vary from place to place in what follows)
\be\label{estimate 1}
\|\alpha\wh{Y}\|_{\wHlamo}^{2}+\|m^{\ast}Z\|_{\wHlamd }^{2}
\leq K_{1}\left[\|\eta-U_{T}\|^{2}_{\wh{L}^2_{\lambda}}+\|\alpha U\|_{\wHlamo}^{2}+
\|\alpha^{-1}h(t,\wh{Y}^{2},Z^{2})\|_{\wHlamo }^{2}\right]
\ee
and thus $\|\alpha Y\|_{\wHlamo }^{2} +\|m^{\ast}Z\|_{\wHlamd }^{2} \leq K_{1}\Delta $. We will now show that
\be \label{dece}
\EP \bigg[\sup_{t\in[0,T]}e^{\lambda N_{t}} \wh{Y}^2_t \bigg]\leq K_{1}\Delta.
\ee
Indeed, since
\bde
\wh{Y}_t=\eta-U_{T}-\int_{t}^{T}Z^{\ast}_{u}\,dM_{u} +\int_{t}^{T}
(h^{1}(u, Y^{1}_{u} , Z^{1}_{u} )-h^{2}(u, Y^{2}_{u} , Z^{2}_{u} ))\,dQ_{u},
\ede
an application of the It\^o formula to $e^{\lambda N_{t}}\wh{Y}^2_t$ yields
\begin{align*}
e^{\lambda N_{t}} \wh{Y}_t^{2}&=e^{\lambda N_{T}}(\eta-U_{T})^{2}
-2\int_{t}^{T}e^{\lambda N_{u}}\wh{Y}_{u}Z^{\ast}_{u} \,dM_{u}-\lambda\int_{t}^{T}e^{\lambda N_{u}}\alpha^2_u \wh{Y}_{u}^{2}\, dQ_{u} \\ &\ \ \ +2\int_{t}^{T} e^{\lambda N_{u}}\wh{Y}_{u}(h^{1}(u, Y^{1}_{u} , Z^{1}_{u} )-h^{2}(u, Y^{2}_{u} , Z^{2}_{u} ))\,dQ_{u}.
\end{align*}
The Burkholder-Davis-Gundy inequality and standard calculus yield
\bde
\EP \bigg[\sup_{t\in[0,T]} e^{\lambda N_{t}}\wh{Y}^2_t \bigg] \leq K_{1}\left[\Delta+\|\alpha\wh{Y}\|_{\wHlamo }^{2}
+\|m^{\ast}Z\|_{\wHlamd }^{2}\right].
\ede
From (\ref{estimate 1}), it now follows that \eqref{dece} is valid, which completes the proof of the proposition.
\end{proof}

\subsection{BSDEs with a Uniformly $m$-Lipschitzian Generator}   \label{sect3.2}

Our next goal is to analyze alternative assumptions under which the existence, uniqueness, and comparison theorems for
BSDE \eqref{special BSDE} can be established. To this end, we consider the case where the processes $\rho$ and $\theta$ in (\ref{Lipschitz for the driver}) are bounded, that is, there exists a constant $\wh{L} >0$ such that $0\leq \rho_{t},\theta_{t}\leq \wh{L}$ for all $t\in[0,T]$.  Since the process $Q$ in Assumption \ref{edded} is bounded, under the assumption that the processes $\rho$ and $\theta$ are bounded as well, the process $N$ is bounded and thus the classes of processes and random variables
satisfying the inequalities \eqref{defhh} and \eqref{defhhh} do not depend on the choice of $\lambda$. In other words, the sets $\wHlamd $ and $\wh{L}^{2}_{\lambda}$ are independent of $\lambda$.
Hence we may take $\lambda=0$ and, since the norms $\|\cdot\|_{\wHlamd }$ and $\|\cdot\|_{\wHzerd}$ are equivalent, for our further purposes, the space $\wHlamd $ may be formally identified with $\wHzerd $. Similarly, the norms $\|\cdot\|_{\wh{L}^2_{\lambda}}$ and $\|\cdot\|_{\wh{L}^2_{0}}$ are equivalent in that case, so that we may identify the spaces $\wh{L}^{2}_{\lambda}$ and $\wh{L}^{2}_{0} = L^2 (\rr)$. Finally, we observe that the norms $\|\cdot\|_{\wHzerd }$ and $\|\cdot\|_{\wh{L}^2_{0}}$ are obviously independent of $N$ (thus also of $\rho$ and $\theta$).

To summarize, if the processes $\rho$ and $\theta$ are bounded, then the spaces $\wHlamd$ and $\wh{L}^{2}_{\lambda}$ do not depend on $\lambda,\, \rho$ and $\theta$.  Therefore, if the processes $\rho$ and $\theta$ in the $m$-Lipschitz condition for $h$ are bounded, then we may assume, without loss of generality, that $\rho_{t} =\theta_{t} = \wh{L}$ for all $t\in[0,T]$. Then the process $\alpha = \wh{L}+ \wh{L}^2 $ is constant as well.

\bd \label{assumption 1 for the driver of the special BSDE}
We say that $h$ satisfies the {\it uniform $m$-Lipschitz condition} if there exists a constant $ \wh{L} >0$ such that,
for all $t\in[0,T]$ and all $y_{1},y_{2}\in\rr,\, z_{1},z_{2}\in\rr^d$,
\be \label{Lipschitz for the driverx}
|h(t,y_{1},z_{1})-h(t,y_{2},z_{2})|\leq \wh{L} \big( |y_{1}-y_{2}|+ \| m_{t}^{\ast}(z_{1}-z_{2})\| \big). 
\ee
\ed

In view of Definition \ref{assumption 1 for the driver of the special BSDE} and the preceding discussion regarding
the spaces $\wHlamd$ and $\wHzerd$, we propose the following modification of Definition \ref{defstan}.

\bd \label{defstanx}
We say that the triplet $(h,\eta ,U)$ is an $(m ,\wh{L})$-{\it standard parameter} if: \hfill \break
(i) $h:\Omega \times [0,T] \times \rr \times \rr^{d} \rightarrow \rr$ satisfies the
uniform $m$-Lipschitz condition with a constant $\wh{L}$, \hfill \break
(ii) the process $h(\cdot,0,0)$ belongs to $\wHzero $, \hfill \break
(iii) a random variable $\eta $ belongs to $\wh{L}^{2}_{0}$, \hfill \break
(iv) $U$ is a real-valued, $\gg$-adapted process
such that $U \in \wHzero$ and $U_{T}\in \wh{L}^{2}_{0}$.
\ed

By a rather straightforward application Theorem \ref{1 wellposedness theorem for special BSDE} and Proposition \ref{1 stability of BSDE}, we obtain the following results for BSDEs with generators satisfying the uniform $m$-Lipschitz condition.
The proofs of both results are almost immediate and thus they are omitted.

\begin{theorem} \label{wellposedness theorem for special BSDE}
If $(h,\eta ,U)$ is an $(m ,\wh{L} )$-standard parameter, then BSDE (\ref{special BSDE}) has a unique
solution $(Y,Z)$ such that $(Y, m^{\ast}Z) \in \wHzero  \times \wHzerd $. Moreover, the process $Y-U$ satisfies
\bde 
\EP \bigg[\sup_{t\in[0,T]}(Y_{t}-U_{t})^{2} \bigg] < \infty.
\ede
\end{theorem}

\bp \label{stability of BSDE}
Let $(h^i,\eta^i,U^i)$ be an $(m,\wh{L}_i)$-standard parameter and let $(Y^i,Z^i)$
be the solution to the following BSDE, for $i=1,2$,
\be \label{special BSDE 2x}
\left\{
\begin{array}
[c]{ll}
dY_t^{i}  =Z^{i,\ast}_t \, dM_t - h^{i}(t, Y^{i}_t , Z^{i}_t )\,dQ_{t} + dU_t^{i},\medskip\\
Y_{T}^{i}=\eta^{i}.
\end{array}
\right.
\ee
If we denote $Y=Y^{1}-Y^{2}$, $Z=Z^{1}-Z^{2}$, $U=U^{1}-U^{2}$, $\eta=\eta^{1}-\eta^{2}$ and  $h=h^{1}-h^{2}$, then
\bde 
\EP \Big[\sup_{t\in[0,T]}(Y_{t}-U_{t})^{2}\Big] + \|Y\|_{\wHzero }^{2}
+ \|m^{\ast}Z\|_{\wHzerd }^{2} \leq K_{1}\widetilde{\Delta}
\ede
where $K_1$ is a constant and
\bde
\widetilde{\Delta}=\left[ \|\eta-U_{T}\|^{2}_{\wh{L}^2_{0}}+\|U\|_{\wHzero}^{2}+
\|h(t,Y^{2}-U^{2},Z^{2})\|_{\wHzero }^{2}\right].
\ede
\ep

\subsection{Comparison Theorem: Multi-Dimensional Case}  \label{sect3.3}

Our next goal, and in fact the main motivation for this work, is to extend Theorem \ref{comparison theorem 1} to BSDE (\ref{special BSDE}) driven by a multi-dimensional martingale. It is worth noting that in \cite{MW-2012,T-2008}, the authors established some versions of the comparison theorem for BSDEs with generators satisfying the quadratic growth condition. For this purpose, they
needed to make some additional assumptions. In the paper by Mocha and Westray \cite{MW-2012}, the comparison theorem is proven using the $\theta$-technique under the postulate that $m$, the terminal value $\eta$, and solution $Y^1$ and $Y^2$ have exponential moments of all orders (see Theorem 5.1 in \cite{MW-2012}).  Tevzadze \cite{T-2008} examined the case when the terminal condition is bounded, and he focussed on bounded solutions $Y$ complemented by BMO martingale component. He established the comparison theorem  using the linearization technique under a certain integrability condition imposed on the process $(mm^{\ast})^{-1}\nabla h (\cdot , Y,Z, \wt{Z}) $ (see condition (L.2) in Theorem 2 in \cite{T-2008}). Let us remark that even if the generator is uniformly $m$-Lipschitzian, his result requires the trace $\text{Tr}[(mm^{\ast})^{-1}]$ to be bounded, which is used the ensure the validity of condition (L.2) (see Remark on page 12 in \cite{T-2008}). We will discuss this condition in more detail later on (see Assumption \ref{assumption 2 for m}).

In our framework, when dealing with the BSDEs with a uniformly $m$-Lipschitzian generator, we do not need the boundedness of $\text{Tr}[(mm^{\ast})^{-1}]$ since, by using the linearization technique, we can obtain comparison theorem under standard assumptions. The crucial difference is that in the proof of Theorem \ref{comparison theorem 2x} we take the `density process' $q$ different to the one employed in \cite{T-2008}.
Throughout this subsection, we work under Assumptions \ref{edded}--\ref{assumption standard generator} complemented by
the following standing assumption.

\bhyp \label{inver}
The matrix $m_{t}$ in representation \eqref{mm33} of the process $\langle M\rangle$ is invertible for all $t \in [0,T]$.
\ehyp

The invertibility of $m_{t}$ allows us to define the auxiliary function $\wh{h}(t,y,z):=h(t,y,m^{-1}_{t}z)$ where $h$ is an arbitrary generator satisfying the uniform $m$-Lipschitz condition. Then the function $\wh{h}$ is uniformly Lipschitzian, since \eqref{Lipschitz for the driverx} entails that
\be \label{Lipschitz m for the driver}
|\wh{h}(t,y_{1},z_{1})-\wh{h}(t,y_{2},z_{2})|\leq \wh{L} \big( |y_{1}-y_{2}|+ \| z_{1}-z_{2}\| \big).
\ee
Let us denote $\wt{Z}_{t,i}^{k}=(m_t^{\ast}Z_t^{k})_{i}$ for $k=1,2$ and $i=1,2,\ldots,d$. For every $j=2,\ldots,d$ and $t \in [0,T]$, we define the following processes
\begin{align*}
&\delta_{Y} \wh{h}_{t}=\frac{\wh{h}(t,Y^{1}_{t},\wt{Z}^{1}_{t})-\wh{h}(t,Y^{2}_{t},\wt{Z}^{1}_{t})}
{Y^{1}_{t}-Y^{2}_t}\, \I_{\{Y^{1}_{t}\neq Y^{2}_t\}}, \\
&\delta_{\wt{Z}_{1}} \wh{h}_{t}=\frac{\wh{h}(t,Y^{2}_{t},\wt{Z}^{1}_{t,1},\wt{Z}^{1}_{t,2},\ldots,\wt{Z}^{1}_{t,d})
-\wh{h}(t,Y^{2}_{t},\wt{Z}^{2}_{t,1},\wt{Z}^{1}_{t,2},\ldots,\wt{Z}^{1}_{t,d})}{\wt{Z}^{1}_{t,1}-\wt{Z}^{2}_{t,1}}
\, \I_{\{\wt{Z}^{1}_{t,1}\neq \wt{Z}^{2}_{t,1}\}},\\
&\delta_{\wt{Z}_{j}}\wh{h}_{t}=\frac{\wh{h}(t,Y^{2}_{t},\ldots,\wt{Z}^{2}_{t,j-1},\wt{Z}^{2}_{t,j},
\wt{Z}^{1}_{t,j+1},\ldots,\wt{Z}^{1}_{t,d})
- \wh{h}(t,Y^{2}_{t},\ldots,\wt{Z}^{2}_{t,j-1},\wt{Z}^{1}_{t,j},
\wt{Z}^{1}_{t,j+1},\ldots,\wt{Z}^{1}_{t,d})}{\wt{Z}^{1}_{t,j}-\wt{Z}^{2}_{t,j}}
\, \I_{\{\wt{Z}^{1}_{t,j}\neq \wt{Z}^{2}_{t,j}\}},
\end{align*}
and write $\delta_{\wt{Z}}h_{t}:=(\delta_{\wt{Z}_{1}}h_{t},\ldots,\delta_{\wt{Z}_{d}}h_{t})$.
We are now in a position to establish the following comparison theorem in which Assumption \ref{assumption 2 for m} is not postulated.

\begin{theorem} \label{comparison theorem 2x}
We postulate that Assumption \ref{edded} holds with a $\gg$-progressively measurable process $m$ and
Assumptions \ref{assumption standard generator}--\ref{inver} are valid. We consider the following two BSDEs, $i=1,2$,
\be \label{special BSDEi}
\left\{
\begin{array}
[c]{ll}
dY_t^{i} =Z^{i,\ast}_t \, dM_t - h^{i}(t, Y^{i}_t , Z^{i}_t )\,dQ_{t} + dU_t^{i},\medskip\\
Y_{T}^{i}=\eta^{i}.
\end{array}
\right.
\ee
Assume that: \hfill \break
(i) the triplet $(h^i,\eta^i ,U^i)$ is an $(m ,\wh{L}_i)$-standard parameter for $i=1,2$, \hfill \break
(ii) the processes $h^i(\cdot,\cdot,y,z),\, i=1,2$ are $\gg$-progressively measurable for every fixed $(y,z)\in\rr \times \rr^{d}$, \hfill \break
(iii) $U^1$ and $U^2$ are $\gg$-progressively measurable processes such that the process $U^{1}-U^{2}$ is decreasing. \hfill \break
If  $\eta^{1}\ge\eta^{2}$ and $h^{1}(\cdot,Y^{2}, Z^{2})\ge h^{2}(\cdot ,Y^{2}, Z^{2}),\, \P\otimes \Leb$-a.e., then $Y^{1}_{t}\ge Y^{2}_{t}$ for every $t\in[0,T]$.
\end{theorem}

\begin{proof}
Since $(h^i,\eta^i ,U^i)$ is an $(m ,\wh{L}_i )$-standard parameter for $i=1,2$,
BSDE (\ref{special BSDEi}) has a unique solution $(Y^i,Z^i)$ such that $(Y^i, m^{\ast}Z^i)\in\wHzero  \times \wHzerd$.
Let us denote
$$
Y=Y^{1}-Y^{2},\ Z=Z^{1}-Z^{2},\ \eta=\eta^{1}-\eta^{2},\ U=U^{1}-U^{2} ,
$$
and $h=h^{1}(\cdot ,Y^{2} , Z^{2})-h^{2}(\cdot ,Y^{2} , Z^{2})$.
Also, let us write $a=\delta_{Y}\wh{h}^{1}$ and $b=\delta_{\wt{Z}} \wh{h}^{1}$. Noticing that $\wh{h}(t,y,m_t^{\ast}z)=h(t,y,z)$, we deduce that the pair $(Y,Z)$ solves the following linear BSDE
\bde
\left\{
\begin{array}
[c]{ll}
dY_t =Z_t^{\ast} \, dM_t - (a_{t}Y_{t}+b_{t}m^{\ast}_{t}Z_{t}+h_{t})\,dQ_{t}+dU_{t},\medskip\\
Y_{T}=\eta.
\end{array}
\right.
\ede
Since assumption (i) implies that the generator $h^{1}$ satisfies the uniform $m$-Lipschitz condition \eqref{Lipschitz for the driverx} with a constant $\wh{L}^1$, it is clear that the function $\wh{h}^{1}$ is uniformly Lipschitzian with the same constant and thus the processes
$a$ and $b$ are bounded, specifically, $|a_t|\leq \wh{L}_1$ and $\|b_t \|\leq \sqrt{d} \wh{L}_1$ for all $t \in [0,T]$.

Let us write $c := m m^{\ast }$ and let us define the `density process' $q$ by
the following expression
\bde 
q_{t}:=\exp\left\{\int_{0}^{t}a_{u}\,dQ_{u}+\int_{0}^{t}b_{u} c_{u}^{-\frac{1}{2}}\, dM_{u}
-\frac{1}{2}\int_{0}^{t}\|b_{u}\|^{2}\, dQ_{u}\right\}.
\ede
From equation (\ref{mm33}), we have $d\langle M\rangle_{t}= c_t \, dQ_t$ and thus
\bde
\Big \langle\int_{0}^{\cdot}b_{u} c_u ^{-\frac{1}{2}}\, dM_{u} \Big\rangle_T
=\int_{0}^{T}\text{Tr}\Big[\big(b_{u}c_u ^{-\frac{1}{2}}\big)^{\ast}
b_{u}c_u ^{-\frac{1}{2}} c_u \Big] \, dQ_u
=\int_{0}^{T}\|b_{u}\|^{2}\, dQ_u\leq d \wh L_1^{2}Q_T.
\ede
Recall that the increasing process $Q$ is assumed to be bounded. Hence $q$ is a strictly positive process and, from Novikov's criterion, the inequality
\be \label{eqqq}
\EP \bigg[ \sup_{t\in[0,T]}q^2_{t} \bigg] < \infty
\ee
is valid. Moreover, an application of the It\^o formula yields
\begin{align*}
dq_{t}&=q_{t}b_{t}c_{t}^{-\frac{1}{2}}\,dM_{t}+q_{t}\Big(a_{t}-\frac{1}{2}\,\|b_{t}\|^{2}\Big)\,dQ_{t}+\frac{1}{2}\,q_{t}\text{Tr}
\Big[ \big(b_{t}c_{t}^{-\frac{1}{2}}\big)^{\ast}b_{t}c_{t}^{-\frac{1}{2}}
c_{t}\Big]\, dQ_{t}\\ &=q_{t}b_{t}c_{t}^{-\frac{1}{2}}\,dM_{t}+q_{t}a_{t}\,dQ_{t},
\end{align*}
and thus, by another application of the It\^o formula, we obtain
\begin{align*}
d (q_{t}Y_t) & =q_{t}Z_t^{\ast} \, dM_t - q_{t}(a_{t}Y_{t}+b_{t}m^{\ast}_{t}Z_{t}+h_{t})\,dQ_{t}+q_{t}\, dU_{t}\\
&\quad\mbox{}+Y_{t}q_{t}b_{t}c_t^{-\frac{1}{2}}\, dM_{t}+Y_{t}q_{t}a_{t}\, dQ_{t}+d\langle Y, q\rangle_t \\
& =q_{t}\big( Z_t^{\ast}+Y_{t}b_{t}c_{t}^{-\frac{1}{2}} \big)\, dM_t-q_{t}(b_{t}m^{\ast}_{t}Z_{t}+h_{t})\,dQ_{t}+q_{t}\, dU_{t}+q_{t}\text{Tr}\left[Z_{t}b_{t} c_{t}^{-\frac{1}{2}} c_{t}\right]dQ_{t}\\
& =q_{t}\big( Z_t^{\ast}+Y_{t}b_{t}c_{t}^{-\frac{1}{2}}\big)\, dM_t-q_{t}h_{t}\, dQ_{t}+q_{t}\, dU_{t}+q_{t}(b_{t}m_t^{\ast}Z_{t}-b_{t}c_{t}^{\frac{1}{2}}Z_{t})\,dQ_{t}\\
& =q_{t}\big( Z_t^{\ast}+Y_{t}b_{t}c_{t}^{-\frac{1}{2}} \big)\, dM_t-q_{t}h_{t}\, dQ_{t}+q_{t}\, dU_{t}
\end{align*}
where the last equality holds since $m^{\ast}_t=(m_{t}m_{t}^{\ast})^{\frac{1}{2}}=c_t^{\frac{1}{2}}$.
Since $(Y,m^{\ast}Z)\in \wHlamo \times \wHlamd $ and inequality \eqref{eqqq} holds,  the process $\wh M$ given by
\be \label{defmmm}
\wh M_t := q_{t}Y_{t}-q_{0}Y_{0}+\int_{0}^{t}q_{u}h_{u}\,dQ_{u}-\int_{0}^{t}q_{u}\,dU_{u}
= \int_{0}^{t}q_{u}\big(Z_{u}^{\ast}+Y_{u}b_{u}c_{u}^{-\frac{1}{2}}\big)\, dM_{u}
\ee
is a local martingale. From Theorem \ref{wellposedness theorem for special BSDE}, we also have that, for $i=1,2$,
\be \label{eqqYU}
\EP \bigg[\sup_{t\in[0,T]}(Y_{t}^{i}-U_{t}^{i})^{2}\bigg]<\infty,
\ee
and thus we may check that $\wh M$ is a uniformly integrable martingale (see the last part of the proof).
Consequently, it is equal to the conditional expectation of its terminal value, which in turn implies that
\bde
q_{t}Y_{t}=\EP \bigg[\, q_{T}\eta+\int_{t}^{T}q_{u}h_{u}\,dQ_{u}-\int_{t}^{T}q_{u}\,dU_{u}\, \Big|\, \mathcal{G}_{t}\bigg]\ge0 .
\ede
Since we assumed that $\eta=\eta^{1}-\eta^{2}\geq 0,\, h_t =h^{1}(t,Y^{2} , Z^{2})-h^{2}(t,Y^{2} , Z^{2}) \geq 0 $ for all $t \in [0,T]$ and the process $U$ is decreasing with $U^1_0-U^2_0=0$, we conclude that the inequality $Y_{t}\ge0$ holds for all $t \in [0,T]$.

To complete the proof, it now remains to demonstrate that the local martingale $\wh M$, which is given by  \eqref{defmmm},
is uniformly integrable. Let us first consider the term $q_{t}Y_{t}$. Since $U$ is a decreasing process, we have that
$|U_{t}|\leq |U_{T}|+|U_{0}|$ and thus, since $U_{T}$ is assumed to belong to $\wh{L}^{2}_{0}$,
\be \label{eqqU}
\EP \bigg(\sup_{t\in[0,T]}|U_{t}|^{2}\bigg)\leq 2\EP \left(|U_{T}|^{2}+|U_{0}|^{2}\right)<\infty.
\ee
Then, by combining (\ref{eqqq}) with (\ref{eqqU}), we obtain
\begin{align*}
\EP \bigg(\sup_{t\in[0,T]}|q_{t}Y_{t}|\bigg)& \leq \bigg[\EP \bigg(\sup_{t\in[0,T]}|q_{t}|^{2}\bigg)\bigg]^{1/2}\bigg[\EP \bigg(\sup_{t\in[0,T]}|Y_{t}|^{2}\bigg)\bigg]^{1/2}\\
& \leq 2 \bigg[\EP \bigg(\sup_{t\in[0,T]}|q_{t}|^{2}\bigg)\bigg]^{1/2}\bigg[\EP \bigg(\sup_{t\in[0,T]}|Y_{t}-U_{t}|^{2}\bigg)+\EP \bigg(\sup_{t\in[0,T]}|U_{t}|^{2}\bigg)\bigg]^{1/2}<\infty
\end{align*}
where we also used \eqref{eqqYU} to establish the last inequality.

Now let us consider the integral $\int_{0}^{t}q_{u}h_{u}\,dQ_{u}$.
Since the generators $h^1$ and $h^2$ satisfy the uniform $m$-Lipschitz condition, there exists some constant $K$,
which may vary from line to line in the following discussion, such that the process
$h_t = h^{1}(t,Y^{2} , Z^{2})-h^{2}(t,Y^{2} , Z^{2})$ satisfies
\[
|h_{t}|\leq K \big( |h^{1}(t,0,0)|+|h^{2}(t,0,0)|+|Y^{2}_{t}|+|m_t^{\ast}Z_{t}^{2}|\big).
\]
Since $(Y^2, m^{\ast}Z^2)\in \wHzero  \times \wHzerd $ and $h^{i}(\cdot,0,0)\in\wHzero $, we see that $h\in\wHzero $.
Therefore, using the boundedness of $Q$, we get
\begin{align*}
\EP \bigg(\sup_{t\in[0,T]}\Big|\int_{0}^{t}q_{u}h_{u}\,dQ_{u}\Big| \bigg)&\leq \EP \bigg(\sup_{t\in[0,T]}\bigg(\int_{0}^{t}|q_{u}h_{u}|^{2}\,dQ_{u}\bigg)^{1/2}Q_{T}^{1/2}\bigg)
\leq  K\EP\bigg(\int_{0}^{T}|q_{u}h_{u}|^{2}\,dQ_{u}\bigg)^{1/2}\\
& \leq K\EP \bigg(\sup_{t\in[0,T]}|q_{t}|\bigg(\int_{0}^{T}|h_{u}|^{2}\,dQ_{u}\bigg)^{1/2}\bigg)\\
& \leq K\bigg[\EP \bigg(\sup_{t\in[0,T]}|q_{t}|^{2}\bigg)\EP \bigg(\int_{0}^{T}|h_{u}|^{2}\,dQ_{u}\bigg)\bigg]^{1/2}
<\infty
\end{align*}
where the last inequality holds in view of \eqref{eqqq} and the previously established property that $h\in\wHzero $.

Finally, we focus on the term $\int_{0}^{t}q_{u}\,dU_{u}$. From  (\ref{eqqq}) and (\ref{eqqU}), we have
\begin{align*}
\EP \bigg(\sup_{t\in[0,T]}\Big|\int_{0}^{t}q_{u}\,dU_{u}\Big|\bigg)&\leq \EP \bigg(|U_T -U_{0}| \sup_{t\in[0,T]}|q_{t}|\bigg) \leq \bigg[\EP |U_T-U_{0}|^{2} \, \EP \bigg(\sup_{t\in[0,T]}|q_{t}|^{2}\bigg)\bigg]^{1/2}
<\infty.
\end{align*}
Consequently, from the definition of $\wh{M}$, we obtain
\[
\EP \bigg(\sup_{t\in[0,T]}|\wh{M}_{t}|\bigg)<\infty,
\]
which implies that $\wh M$ is a uniformly integrable martingale.
\end{proof}

\brem \label{remark for comparison theorem 2}
If $\eta^{1}\ge\eta^{2}$ and $h^{1}(\cdot ,Y^{1}, Z^{1})\ge h^{2}(\cdot ,Y^{1}, Z^{1}),\, \P\otimes \Leb$-a.e.,
then one can also prove that $Y^{1}_{t}\ge Y^{2}_{t}$  for every $t\in[0,T]$.
\erem

\brem \label{financial remark for comparison theorem}
In the above proof, we needed to ensure that the stochastic integral
\be \label{inhh}
\int_{0}^{t}b_{u}(m_{u}m^{\ast}_{u})^{-\frac{1}{2}}\, dM_{u},
\ee
and thus also the process $q$, are well defined.  From the monograph by Karatzas and Shreve \cite{KS-1998} (see Chapter 3, Definition 2.9), we know that the stochastic integral \eqref{inhh} is well defined when the processes $b$ and $m$ are $\gg$-progressively measurable. For this reason, we require that the processes $m,\, U^{i}$, as well as the process $h^{i}(\cdot,y,z)$,
for any fixed $(y,z)\in\rr \times \rr^{d}$, are $\gg$-progressively measurable.
Furthermore,  from \cite{KS-1998} (Chapter 3, Remark 2.11), if  $Q_{t}=t$ in (\ref{mm33}), then the stochastic integral \eqref{inhh} is well defined, provided that the processes $b$ and $m$ are $\gg$-adapted (not necessarily $\gg$-progressively measurable). Therefore, when $Q_{t}=t$, the adaptedness is sufficient. Let us mention that in our financial applications, we usually  have $Q_{t}=t$ (see \cite{BR-2014,NR-2014,NR-2014a}). In Section \ref{sect5}, we provide an example of a financial market model, in which we may take $Q_{t}=t$.
\erem

\section{BSDEs with a Lipschitzian Generator}   \label{sect4}

It is fair to acknowledge that the concept of a (uniformly) $m$-Lipschitzian generator,
although very convenient for the mathematical analysis of BSDE \eqref{special BSDE}, can be seen as
somewhat  artificial from a more practical point of view. Indeed, typically a particular class
of BSDEs arises in a natural way when solving problems within a given framework, so the shape of
the BSDE and its generator is imposed by the problem at hand, rather than arbitrarily postulated.
The goal of this section is to provide a link between BSDEs \eqref{special BSDE} with uniformly $m$-Lipschitzian
generators and some classes of BSDEs arising in various applications to stochastic optimal control and
financial mathematics. We will also make some pertinent comments on solvability of BSDEs given by (\ref{multid BSDE}).

\subsection{BSDEs with a Uniformly Lipschitzian Generator}   \label{sect4.1}

In the case of BSDEs driven by a Brownian motion, it is common to suppose that the generator is uniformly Lipschitzian,
as we also postulated in the case of BSDEs driven by a one-dimensional martingale (see condition \eqref{uniformly Lipschitz for the driver}). By contrast, most of existing studies of BSDEs driven by a multi-dimensional martingale hinge on the postulate that the generator satisfies some form of the $m$-Lipschitz condition. The latter choice seems to be motivated mainly by mathematical convenience.

Since our comparison theorem requires the generator to be uniformly $m$-Lipschitzian, the following natural question thus arises. Suppose that a generator $h$ satisfies the uniform Lipschitz condition \eqref{uniformly Lipschitz for the driver}. Does this mean that $h$ satisfies the uniform $m$-Lipschitz condition \eqref{Lipschitz for the driverx} as well? To answer this question, we need to take a closer look on the term $m$ appearing in factorization (\ref{mm33}). In the case of a general process $m$, the following assumption may be introduced.

\bhyp \label{assumption 2 for m}
There exists a constant $K_m>0$ such that, for all $t\in[0,T]$,
\be \label{mmu}
\norm m_{t}\norm+\norm(m_{t}m_{t}^{\ast})^{-\frac{1}{2}}\norm\leq K_m .
\ee
\ehyp

As shown in the next lemma, condition \eqref{mmu} is a convenient way of ensuring that the uniform $m$-Lipschitz condition \eqref{Lipschitz for the driverx} for a generator $h$ holds. It is fair to acknowledge, however, that
condition \eqref{mmu} has a shortcoming that it is not satisfied in a typical market model and thus its usefulness is somewhat limited in the context of problems arising in financial mathematics.

\bl \label{lemma for different generator}
Under Assumption \ref{assumption 2 for m}, the generator $h$ satisfies the uniform Lipschitz condition \eqref{uniformly Lipschitz for the driver} if and only if it satisfies the uniform $m$-Lipschitz condition \eqref{Lipschitz for the driverx}.
\el

\begin{proof}
It is clear that when \eqref{Lipschitz for the driverx} is combined with \eqref{mmu} then condition \eqref{uniformly Lipschitz for the driver} holds for some constant $L$. To show that the converse implication is valid as well, we assume that (\ref{uniformly Lipschitz for the driver}) holds with a constant $L$. Under Assumption \ref{assumption 2 for m}, we have that $\norm (m_{t}m_{t}^{\ast})^{-\frac{1}{2}}\norm\leq K_m$, which implies that the eigenvalues of $m$ are all greater than or equal to $\lambda_m :=1/K_m$. Consequently, we obtain
\bde
\|m_{t}^{\ast}z\|^{2}=z^{\ast}m_{t}m_{t}^{\ast}z\ge \Lambda_m \|z\|^{2}
\ede
where $\Lambda_m = \lambda^2_m$. Therefore, upon setting $\widehat{L}= L \max (1,K_m)$, we conclude that (\ref{Lipschitz for the driverx}) is valid.
\end{proof}

It is clear that Assumption \ref{assumption 2 for m} implies that there exists a constant $k_m$ such that
$\norm m_{t} \norm \ge k_m >0$ for all $t\in[0,T]$. Then the random variable
$\langle M\rangle_{T}$ given by \eqref{mm33} is bounded, since the process $Q$ was assumed to be bounded.

Moreover, from above lemma, we know that under Assumption \ref{assumption 2 for m}, all the results in Section \ref{sect3} hold for BSDEs (\ref{special BSDE}) with a uniformly Lipschitzian generator.

We argue that Assumption \ref{assumption 2 for m} would be also convenient when dealing with BSDE (\ref{multid BSDE}), which has the following form
\be \label{multid BSDExx}
dY_t  = Z_t^{\ast} \, dM_t -  \Iast d\langle M\rangle_{t}\, \hdd(t, Y_t , Z_t ) + dU_t ,\quad Y_{T}= \eta ,
\ee
with the generator $\hdd: \Omega \times[0,T] \times \rr \times \rr^d \rightarrow \rr^d$ satisfying the
measurability Assumption~\ref{assumption standard generator}.

\newpage

To the best of our knowledge, due to its complexity, the BSDE of this shape was not yet studied in detail in the existing literature. Under Assumptions \ref{edded} and \ref{assumption 2 for m}, BSDE (\ref{multid BSDExx}) may be represented as follows
\be \label{multid BSDE 2}
dY_t  = Z_t^{\ast} \, dM_t -  \Iast m_{t}m^{\ast}_{t}\, \hdd(t, Y_t , Z_t )\, dQ_{t}\, + dU_t ,\quad Y_{T}= \eta ,
\ee
where we make the assumption that the $\rr^d$-valued generator $\hdd$ is uniformly Lipschitzian, specifically,
there exists a constant $L_d>0$ such that, for all $t\in[0,T]$ and all $y_{1},y_{2}\in\mathbb{R},\, z_{1},z_{2}\in\mathbb{R}^d$,
\[
\|\hdd(t,y_{1},z_{1})-\hdd(t,y_{2},z_{2})\|\leq L_d \big( |y_{1}-y_{2}|+ \|z_{1}-z_{2}\| \big).
\]
The following proposition shows that the classes of BSDEs (\ref{special BSDE}) and
(\ref{multid BSDE 2}) with uniformly Lipschitzian generators are essentially equivalent.

\bp
Under Assumptions \ref{edded} and \ref{assumption 2 for m}, the problem of solving BSDE (\ref{multid BSDE 2}) with a uniformly Lipschitzian $\rr^d$-valued generator $\hdd$ is essentially equivalent to solving BSDE (\ref{special BSDE}) with a uniformly Lipschitzian real-valued generator $h$.
\ep

\begin{proof}
Assume first that the $\rr^d$-valued generator $\hdd$ in BSDE \eqref{multid BSDE 2} is uniformly Lipschitzian.
Then we define the associated real-valued generator $h$ by setting $h(t,y,z):=\Iast m_{t}m^{\ast}_{t}\, \hdd(t, y , z)$. In view of the inequality $\norm m \norm \leq K_m$ (see Assumption \ref{assumption 2 for m}), it is clear that the generator $h$ is also uniformly Lipschitzian. This means that BSDE (\ref{multid BSDE 2}) can be reduced to BSDE (\ref{special BSDE}) with a uniformly Lipschitzian generator $h$. Hence if the answer to the well-posedness problem for BSDE (\ref{special BSDE}) is positive, then the same property is enjoyed by BSDE (\ref{multid BSDE 2}).

Conversely, we observe that to any real-valued generator $h$ we may associated the $\rr^d$-valued generator $\hdd$.
To this end, we may simply take $\hdd:=(mm^{\ast})^{-1}(h,0,\ldots,0)^{\ast}$. Since now the real-valued generator $h$ is assumed to be uniformly Lipschitzian, in view of the inequality $\norm (mm^{\ast})^{-\frac{1}{2}}\norm \leq K_m$, we conclude that the associated $\rr^d$-valued generator $\hdd$ is also uniformly Lipschitzian. Therefore, if a result yielding the existence and uniqueness of a solution to BSDE (\ref{multid BSDE 2}) with a  uniformly Lipschitzian generator $\hdd$ is available, then this result covers BSDE (\ref{special BSDE}) as well.
\end{proof}

Let us now consider the comparison theorem for BSDEs (\ref{multid BSDE 2}). Suppose that the generator $\hdd$ is uniformly Lipschitzian and, in addition, $m, U$ are $\gg$-progressively measurable and $\hdd(\cdot,\cdot,y,z)$ is $\gg$-progressively measurable for every fixed $(y,z)\in\mathbb{R}\times\mathbb{R}^{d}$. Then one can establish a version of the comparison theorem for BSDEs (\ref{multid BSDE 2}) by either using Theorem \ref{comparison theorem 2x} or directly from the results of
Tevzadze \cite{T-2008} (in the latter case, by employing also the boundedness of $\langle M\rangle_{T}$).  In such case,
the assumption that $h^{1}(\cdot,Y^{2}, Z^{2})\ge h^{2}(\cdot ,Y^{2}, Z^{2}),\, \P\otimes \Leb$-a.e., should
be replaced by
\bde
\Iast mm^{\ast} \hdd^{1}(\cdot,Y^{2}, Z^{2})\ge \Iast m m^{\ast} \hdd^{2}(\cdot ,Y^{2}, Z^{2}),
\quad \P\otimes \Leb-\text{a.e.},
\ede
which seems to be cumbersome to verify. For the reasons explained above, we leave this task for a future study,
and we henceforth focus on alternative assumptions on a generator that, as will be shown in Section \ref{sect5},
are satisfied by BSDEs arising in market models with funding costs.

\subsection{BSDEs with a Uniformly $\mathbb{\XX}$-Lipschitzian Generator}   \label{sect4.2}

We stress that Assumption \ref{assumption 2 for m} covers the case when the process
$m$ in (\ref{mm33}) is not explicitly known. In a typical applications, we have more information about the
shape of the generator $h$ and perhaps also the process $m$.
The motivation for the setup studied in this subsection comes from various applications of BSDEs in financial mathematics (see, e.g., the seminal paper by El Karoui et al. \cite{EPQ-1997}). To specify our setup, we start by defining the matrix-valued process
$\XX$, which is given by
\be \label{eqss}
\XX_{t}:=
\begin{pmatrix}
X^{1}_{t} & 0 & \ldots & 0 \\
0 & X^{2}_{t} & \ldots & 0 \\
\vdots & \vdots & \ddots & \vdots\\
0 & 0 & \ldots & X^{d}_{t}
\end{pmatrix}
\ee
where the auxiliary processes $X^i,\, i=1,2,\dots ,d$ are assumed to be $\gg$-adapted.

\newpage

The auxiliary processes $X^i,\, i=1,2,\dots ,d$ arise naturally in some applications, so there their choice is not arbitrary,
but depends on a particular application at hand. In some instances, it may happen that $X^i = M^i$ for all $i$ but, typically, the processes $X^i$ and $M^i$ will be different, albeit they are usually closely related. For instance, the martingales $M^i,\, i=1,2,\dots ,d$ and the auxiliary processes $X^i,\, i=1,2,\dots ,d$ may be obtained from a predetermined family of some underlying processes either through integration or by solving stochastic differential equations driven by processes from this family.
For an explicit illustration of the last statement, we refer to Section~\ref{sect5}.

In this subsection, we postulate that the generator $h$ can be represented as $h(t,y,z)= g (t,y,\XX_{t}z )$ for some
function $g:\Omega \times[0,T] \times \rr \times \rr^{d} \rightarrow \rr$
satisfying Assumption \ref{assumption standard generator}. Then BSDE \eqref{special BSDE} can be represented as follows
\be \label{special BSDEmm}
\left\{
\begin{array}
[c]{ll}
dY_t  = Z_t^{\ast} \, dM_t - g(t,Y_t,\XX_{t}Z_t )\, dQ_{t} + dU_t,\medskip\\
Y_{T}=\eta .
\end{array}
\right.
\ee

Suppose that equation \eqref{mm33} holds for some $\rr^{d \times d}$-valued, $\gg$-adapted process $m$
and the generator $h(t,y,z)= g (t,y,\XX_{t}z)$ where the function $g$ satisfies the uniform
Lipschitz condition. We are now going to address following natural question: under which assumptions about $M,m$ and $\mathbb{\XX}$, the generator $h$ satisfies the (uniform) $m$-Lipschitz condition?

We first observe that to ensure that the generator $h$ satisfies the $m$-Lipschitz condition, it suffices to postulate that a strictly positive lower bound for the norm $\norm m \norm$ exists. However, to ensure that $h$ satisfies the uniform $m$-Lipschitz condition, we still need to postulate, in addition, that the processes $X^{i}$ are bounded as well.

To sum up, if we postulate that the process $\norm m \norm$ is bounded away from zero and the processes $X^{i},\, i=1,2,\ldots,d$ are bounded, then the generator $h$ is both uniformly Lipschitzian and uniformly $m$-Lipschitzian.  Obviously, the boundedness of the driving martingale is a very restrictive condition, since it is unlikely to be satisfied in most applications.
Fortunately, in a typical application, one has more information about the driving martingales, which can be used to describe
 a suitable class of generators. This observation allows us to introduce Assumption \ref{assumption 3 for m} and to argue that the comparison theorem can still be applied, despite the fact that Assumption \ref{assumption 2 for m} fails to hold.
In Assumption \ref{assumption 3 for m},  we will employ the following standard definition of ellipticity.

\bd \label{non-degenerate}
We say that an $\rr^{d \times d}$-valued process $\gamma$ satisfies the {\it ellipticity} condition if there exists a constant $\Lambda>0$ such that
\be \label{elli}
\sum_{i,j=1}^{d}\left(\gamma_{t}\gamma^{\ast}_{t}\right)_{ij}a_{i}a_{j}\ge \Lambda \|a\|^{2} 
\ \text{ for all } a\in \rr^{d} \text{ and } t\in[0,T].
\ee
\ed

For the justification of the next assumption in the context of financial models driven by a multi-dimensional Brownian motion,
see Section \ref{sect5}.

\bhyp  \label{assumption 3 for m}
The $\rr^{d \times d}$-valued, $\gg$-adapted process $m$ in equation \eqref{mm33} is given by
\[
m_{t}m_{t}^{\ast}= \XX_{t}\gamma_{t}\gamma_{t}^{\ast}\mathbb{\XX}_{t}
\]
where $\gamma = [\gamma^{ij}]$ is a $d$-dimensional square matrix of $\gg$-adapted processes satisfying
the ellipticity condition \eqref{elli}.
\ehyp

The following definition is natural when dealing with a generator $h(t,y,z)= g ( t,y, \mathbb{\XX}_{t} z )$.

\bd
 We say that a generator $h$ satisfies the {\it uniform $\XX$-Lipschitz condition} if there exists a constant $\wt{L}$ such that, for every $y_1, y_2 \in \rr$ and $z_1 , z_2 \in \rr^d$,
\be \label{special generator}
|h(t,y_1,z_1)-h(t,y_2,z_2)|\leq \wt{L} \big( |y_1 - y_2|+ \|\XX_{t} (z_1-z_2) \|\big).
\ee
\ed

Let us note that condition \eqref{special generator} is equivalent to the following condition: there exists a constant
$\wt{L}_0$ such that, for every $y_1, y_2 \in \rr$ and $z_1 , z_2 \in \rr^d$,
\be \label{special generatorc}
|h(t,y_1,z_1)-h(t,y_2,z_2)|\leq \wt{L}_0 \Big( |y_1 - y_2|+\sum_{i=1}^{d}|X^{i}_{t}(z_1^{i}-z_2^{i}) |\Big)
\ee
where $z_k = (z^1_k, z^2_k,\dots , z^d_k)^{\ast}$ for $k=1,2$. It is worth emphasizing that this condition is frequently satisfied by generators of BSDEs are obtained by analyzing the dynamics of trading strategies (see, for instance, the generator $\wt{f}_l$ given by \eqref{drift function lending}).

The next lemma shows that a combination of conditions  \eqref{elli} and \eqref{special generator}
ensures that a generator is uniformly $m$-Lipschitzian.

\bl \label{lemma special m generator property}
If Assumption \ref{assumption 3 for m} holds and the generator $h$ is uniformly $\XX$-Lipschitzian,
then $h$ is uniformly $m$-Lipschitzian with $\wh{L} = \wt{L} \max \big( 1, \Lambda^{-1/2} \big)$ where
$\Lambda$ is the constant of Definition \ref{non-degenerate}.
\el

\begin{proof}
Assumption \ref{assumption 3 for m} yields, for every $z \in \rr^d$,
\be \label{non-degenerate inequality}
 \| m_{t}^{\ast} z \|^2 = z^{\ast}m_{t}m_{t}^{\ast}z=z^{\ast}\XX_{t}\gamma_{t}\gamma_{t}^{\ast}\XX_{t}z
\geq \Lambda \|\XX_{t} z \|^2 .
\ee
By combining \eqref{special generator} and  (\ref{non-degenerate inequality}), we obtain
\[
|h(t,y_1,z_1)-h(t,y_2,z_2)| \leq \widetilde{L} \big(|y_1- y_2|+ \|\XX_{t} (z_1-z_2) \|\big)
\leq \wt{L} |y-\wt{y}|+ \wt{L} K \| m_{t}^{\ast}(z_1 - z_2) \|,
\]
and thus the generator $h$ satisfies the uniform $m$-Lipschitz condition with $\wh{L}= \wt{L} \max \big( 1, \Lambda^{-1/2} \big)$.
\end{proof}


We are in a position to establish the existence and uniqueness result for BSDE \eqref{special BSDE}
under either of Assumptions \ref{assumption 2 for m} and \ref{assumption 3 for m}. Note that if
Assumption \ref{assumption 2 for m} is postulated, then we assume, in addition, that the processes
$X^i,\, i=1,2,\dots ,d$ are bounded.

\begin{theorem} \label{appendix wellposedness theorem for lending special BSDE}
Assume that the generator $h$ can be represented as $h(t,y,z)= g (t,y,\XX_{t}z)$
where the function $g:\Omega \times [0,T] \times \rr \times \rr^{d} \rightarrow \rr$
satisfies the uniform Lipschitz condition, so that there exists a constant $\bar L$ such that,
for every $y_1, y_2 \in \rr$ and $z_1 , z_2 \in \rr^d$,
\be \label{lipmm}
|g(t,y_{1},z_{1})-g(t,y_{2},z_{2})|\leq \bar{L} \big( |y_{1}-y_{2}|+\|z_{1}-z_{2}\| \big).
\ee
Let the process $h(\cdot,0,0)$ belong to $\wHzero $, the random variable $\eta $ belong to $\wh{L}^{2}_{0}$, and $U$ be a real-valued, $\gg$-adapted process such that $U \in \wHzero$ and $U_{T}\in \wh{L}^{2}_{0}$. Assume that one of the following holds: \hfill \break
(i) the process $m$ satisfies Assumption \ref{assumption 2 for m} and the process $\XX$ is bounded, \hfill \break
(ii) the process $m$ satisfies Assumption \ref{assumption 3 for m} some constant $\Lambda >0 $. \hfill \break
Then BSDE \eqref{special BSDE} has a unique solution $(Y,Z)$ such that $(Y, m^{\ast}Z)\in\wHzero  \times \wHzerd$.
Moreover, the processes $Y$ and $U$ satisfy $$\EP \bigg[\sup_{t\in[0,T]}|Y_{t}-U_{t}|^{2}\bigg]<\infty .$$
\end{theorem}

\begin{proof}
(i) We first postulate that Assumption \ref{assumption 2 for m} holds and the process $X^{i}$ is bounded for every $i=1,2,\ldots,d$. One can deduce that $h$ satisfies the uniform $m$-Lipschitz condition with a constant $\wh{L}$, which depends on the bounds for $X^{i}$ for $i=1,2,\ldots,d$, as well as on the lower bound for $\norm m \norm $.
Therefore, the statement is an immediate consequence of Theorem \ref{wellposedness theorem for special BSDE}.

\noindent (ii) Since $h(t,y,z)= g (t,y,\XX_{t}z)$ where the function $g$
satisfies the uniform Lipschitz condition, it is clear that $h$ satisfies the uniform $\XX$-Lipschitz condition
with $\wt{L}= \bar{L}$. From Assumption \ref{assumption 3 for m} and Lemma \ref{lemma special m generator property},
we deduce that $h$ satisfies the uniform $m$-Lipschitz condition $\wh{L} = \bar{L} \max \big( 1, \Lambda^{-1/2} \big)$. Hence, once again, the assertion follows by an application of Theorem \ref{wellposedness theorem for special BSDE}.
\end{proof}

\newpage

\section{BSDEs in Market Models with Funding Costs} \label{sect5}

We will now demonstrate that the comparison theorem established in Section \ref{sect3} can be applied
to obtain lower and upper bounds on unilateral prices in a market model under funding costs.
For the detailed analysis of issues related to the postulated trading mechanism, the no-arbitrage property of the market,
and the pricing and hedging of a collateralized contract, the reader is referred to \cite{BR-2014,NR-2014,NR-2014a}. In this section, we will focus on the r\^ole of BSDEs in producing inequalities yielding the range of fair bilateral prices.

\subsection{Risky Assets and Funding Accounts} \label{sect5.1}

Let us first recall the following setting of \cite{BR-2014} for the market model.
Throughout the paper, we fix a finite trading horizon date $T>0$ for our model of the financial market.
Let $(\Omega, \cG, \gg , \P)$ be a filtered probability space satisfying the usual conditions of right-continuity and completeness, where the filtration $\gg = (\cG_t)_{t \in [0,T]}$ models the flow of information available to all traders. For convenience, we assume that the initial $\sigma$-field $\cG_0$ is trivial. All processes introduced in what follows are implicitly assumed to be $\gg$-adapted  and any semimartingale is assumed to be c\`adl\`ag.

For $i=1,2, \dots, d$, we denote by $S^i$ the {\it ex-dividend price} of the $i$th risky asset with the {\it cumulative dividend stream} $A^i$. The risk-free {\it lending} (resp., {\it borrowing}) {\it cash account} $B^l$ (resp., $B^b$) is used for unsecured lending (resp., borrowing) of cash. We denote by $B^{i,b}$ {\it funding account} associated with the $i$th risky asset.
The corresponding short-term interest rates $r^{l},r^{b},r^{i,b}$ are non-negative and bounded processes, the bounded processes $B^l$ and $B^r$ satisfy $dB^l_{t}=r^{l}_{t}B^l_{t}\, dt$ and $dB^b_{t}=r^{b}_{t}B^b_{t}\, dt$
with $B^l_{0}= B^r_0= 1$, so that the inequalities $B^l_t \geq 1$ and $B^b_t \geq 1$ hold for all $t \in [0,T]$.
Finally, we also introduce the funding accounts $B^{C,l}$ and $B^{C,b}$ for the margin account represented by a process
 $C$. It is assumed throughout that $0\leq r^l_t\leq r^b_t$ and $r^l_t\leq r^{i,b}_t$ for all $t \in [0,T]$.

For convenience, we introduce the matrix-valued process $\mathbb{S}$ given by
\[
\mathbb{S}_{t}:=
\begin{pmatrix}
S^{1}_{t} & 0 & \ldots & 0 \\
0 & S^{2}_{t} & \ldots & 0 \\
\vdots & \vdots & \ddots & \vdots\\
0 & 0 & \ldots & S^{d}_{t}
\end{pmatrix}.
\]
The discounted cumulative prices of risky assets are given by the following expressions
\bde 
\wt S^{i,l,{\textrm{cld}}}_t =  (B^l_t)^{-1}S^i_t + \int_{(0,t]} (B^l_u)^{-1} \, dA^i_u
\ede
where $A^i$ is the {\it dividend process} of the $i$th risky asset, so that
\be \label{cumulative dividend risk asset price1}
d\wt S^{i,l,{\textrm{cld}}}_t=(B^l_t)^{-1}\big(dS^i_t - r^l_t S^i_t \, dt + dA^i_t\big).
\ee

The following assumption corresponds to Assumptions \ref{mm33} and \ref{assumption 3 for m} (see also Assumption 4.2 in \cite{NR-2014}). Note that here $\PT^l$ is a probability measure equivalent to $\P$ on $(\Omega , \cG_T)$.

\bhyp \label{changed assumption for lending cumulative dividend price}
We postulate that: \hfill \break
(i) the process $\wt S^{l,\textrm{cld}} = ( \wt S^{1,l,\textrm{cld}}, \dots , \wt S^{d,l,\textrm{cld}})^*$ is a continuous, square-integrable, $(\PT^l , \gg)$-martingale and has the predictable representation property with respect to the filtration $\gg$ under~$\PT^l$, \hfill \break
(ii) there exists a $\gg$-adapted, bounded, increasing process $Q$ such that the equality
\be \label{gmma1}
\langle \wt S^{l,\textrm{cld}}\rangle_{t}=\int_{0}^{t}m^{l}_{u}(m_{u}^{l})^{\ast}\,dQ_u
\ee
 holds, where a $\gg$-adapted process $m^{l}$ is such that the $d$-dimensional square matrix $m^{l}$ is invertible and
admits the representation
\be \label{gmma2}
m^{l}(m^{l})^{\ast} = \mathbb{S} \gamma \gamma^{\ast} \mathbb{S}
\ee
where a $d$-dimensional square matrix $\gamma $ of $\gg$-adapted processes satisfies the ellipticity condition~\eqref{elli}.
\ehyp

It is worth noting that Assumption \ref{changed assumption for lending cumulative dividend price} can be easily met when the prices of risky assets are given by the diffusion-type model. For example, we may assume that each risky asset $S^i,\, i=1, 2, \dots  ,d$ has the ex-dividend price dynamics under $\P$ given by
\[
dS^i_t = S^i_t \bigg( \mu^i_t \, dt + \sum_{j=1}^d \sigma^{ij}_t \, dW^j_t \bigg), \quad S^i_0>0 ,
\]
or, equivalently, the $d$-dimensional process $S=(S^{1},\ldots,S^{d})^{\ast}$ satisfies
\[
dS_{t}=\mathbb{S}_{t}(\mu_{t} \, dt+\sigma_{t} \, dW_{t})
\]
where $W = (W^1, \dots , W^d)^{\ast}$ is the $d$-dimensional Brownian motion, $\mu = (\mu^1, \dots , \mu^d)^{\ast}$ is an $\rr^d$-valued, $\ff^W$-adapted process, $\sigma = [\sigma^{ij}]$ is a $d$-dimensional square matrix of $\ff^W$-adapted processes satisfying the ellipticity condition  (see Definition \ref{non-degenerate}). We now set $\gg = \ff^W$ and we recall that the $d$-dimensional Brownian motion $W$ enjoys the predictable representation property with respect to its natural filtration $\ff^W$. Hence this property is shared by the $\PT^{l}$-Brownian motion $\wt{W}$, which is defined by equation \eqref{wtw2}.

Assuming that the corresponding dividend processes are given by $A^i_t = \int_0^t \kappa^i_{u} S^i_{u} \, du $,
we obtain
\bde
d\wt S_t^{i,l,\textrm{cld}}=(B_{t}^{l})^{-1} \big( dS^i_t+ dA^i_t-r_{t}^{l}S_{t}^{i}\, dt \big)
=(B_{t}^{l})^{-1}S^{i}_{t}\bigg( \big( \mu^i_t+\kappa^i_t-r_{t}^{l} \big)\,dt + \sum_{j=1}^d \sigma^{ij}_t \, dW^j_t \bigg).
\ede
If we denote $\mu+\kappa-r^{l}= (\mu^1+\kappa^{1}-r^{l}, \dots , \mu^d+\kappa^{d}-r^{l})^{\ast}$, then the above equation becomes
\bde
d\wt S_t^{l,\textrm{cld}}=(B_{t}^{l})^{-1}\mathbb{S}_{t}\Big( \big(\mu_t+\kappa_t-r_{t}^{l}\big)\,dt+ \sigma_t \, dW_t \Big).
\ede
We set $a:=\sigma^{-1}(\mu+\kappa-r^{l})$ and we define the probability
measure $\PT^{l}$ equivalent to $\P$ on $(\Omega , {\cal F}^W_T)$ by setting
\bde
\frac{d\PT^{l}}{d\P}=\exp\bigg\{-\int_{0}^{T}a_{t}\, dW_{t}-\frac{1}{2}\int_{0}^{T}|a_{t}|^{2}\, dt\bigg\}.
\ede
From the Girsanov theorem, we obtain
\bde
d\wt S_t^{l,\textrm{cld}} = (B_{t}^{l})^{-1}\mathbb{S}_{t} \sigma_t \, d\wt{W}_t
\ede
where the process $\widetilde{W}:=(\widetilde{W}^{1},\widetilde{W}^{2},\ldots,\widetilde{W}^{d})^{\ast}$, which is
given by
\be \lab{wtw2}
d\widetilde{W}_{t} := dW_{t}+a_{t}\, dt = dW_{t}+\sigma_t^{-1}(\mu_t+\kappa_t-r^{l}_{t})\,dt,
\ee
is a Brownian motion under $\PT^{l}$. Therefore, if the processes $\mu,\, \sigma$ and $\kappa$ are bounded, then the process
$\wt S^{l,\textrm{cld}}$ is a continuous, square-integrable, $(\PT^{l}, \ff^W )$-martingale.
In this sense, the probability measure $\PT^{l}$ is an equivalent martingale measure for the present setup.
Furthermore, $\wt S^{l,\textrm{cld}}$ and $m^l$ satisfy conditions \eqref{gmma1}--\eqref{gmma2} with $Q_t := t$ and
$\gamma_t :=(B^{l}_t)^{-1}\sigma_t$ for all $t \in [0,T]$.
Obviously, $m^{l}(m^{l})^{\ast}$ is invertible and thus all conditions in Assumption \ref{changed assumption for lending cumulative dividend price} are satisfied. One should observe that it is natural to identify the processes $\wt S^{l,\textrm{cld}}$  and $\mathbb{S}$ with the processes $M$ and $\XX$ in Section \ref{sect4.2}, respectively.

The following assumption corresponds to Assumptions \ref{mm33} and \ref{assumption 2 for m} (see also Assumption 4.1 in \cite{NR-2014}). Note that it requires the prices of risky assets to be bounded.

\bhyp \label{xx2}
We postulate that: \hfill \break
(i) the process $\wt S^{l,\textrm{cld}} = ( \wt S^{1,l,\textrm{cld}}, \dots , \wt S^{d,l,\textrm{cld}})^*$ is a continuous, square-integrable, $(\PT^l , \gg)$-martingale and has the predictable representation property with respect to the filtration $\gg$ under~$\PT^l$, \hfill \break
(ii) there exists an $\mathbb{R}^{d\times d}$-valued, $\gg$-adapted process $m$ such that \eqref{gmma1} holds
where $m^l(m^l)$ is invertible and there exists a constant $K_m>0$ such that, for all $t\in[0,T]$,
\be \label{mmc2x}
\norm m^l_{t}\norm+\norm(m^l_{t}(m^l_{t})^{\ast})^{-\frac{1}{2}}\norm\leq K_m,
\ee
(iii) the price processes $S^{i},\, i=1,2,\ldots,d$ of risky assets are bounded.
\ehyp

\brem
In the special case where the assets prices are assumed to be uncorrelated, the volatility matrix $\sigma$ is diagonal with the entry $\sigma^{ii}$ denoted as $\sigma^i$, so that the dynamics of the price process of the $i$th risky asset become
\bde
dS^i_t = S^i_t \big( \mu^i_t \, dt +  \sigma^{i}_t \, dW^i_t \big).
\ede
If we postulate that $\mu^i ,\sigma^i$ and $\kappa^i$ are bounded, $\mathbb{F}^W$-adapted processes and the processes $\sigma^i$ are bounded away from zero, specifically, $|\sigma^i_t | > C^i >0$ for all $i$ and $t \in [0,T]$, then Assumption \ref{changed assumption for lending cumulative dividend price} is satisfied.
\erem

\subsection{BSDEs for Unilateral Prices under Funding Costs}    \label{sect5.2}

We henceforth postulate that condition \eqref{gmma1} holds with $Q_t = t$.
Recall from \cite{BR-2014,NR-2014} that the process $A^{C}:=A+C+F^{C}$ models all cash flows from a collateralized contract $(A,C)$. In particular, the process $F^C$, which represents the cumulative interest of margin account, depends on the adopted collateral convention (see Section 4 in \cite{BR-2014} and, in particular, equation  (2.12) in \cite{NR-2014}).  For brevity, we write
\bde
A^{C,l}_t= \int_0^t (B^l_{u})^{-1}\, dA^C_{u}. 
\ede
We say that a contract $(A,C)$ is {\it admissible under} $\PT^l$ if the process $A^{C,l}$ belongs to the space $\wHzero$
and the random variable $A^{C,l}_{T}$ belongs to $\widehat{L}^{2}_{0}$ under~$\PT^l$.

Let the mapping $\wt{f}_l : \Omega  \times [0,T]  \times \rr \times \rr^{d} \to \rr$ be given by (see equation
(2.16) in \cite{NR-2014})
\be \label{drift function lending}
\wt{f}_l( t, y ,z ): = (B^l_t)^{-1} f_l(t,B^l_t y ,z ) -  r^l_t y
\ee
and $f_l : \Omega  \times [0,T]  \times \rr \times \rr^{d} \to \rr$ equals
\bde
f_l(t, y,z)  :=   \sum_{i=1}^d r^l_t z^i S^i_t
- \sum_{i=1}^d r^{i,b}_t( z^i S^i_t )^+
+   r^l_t \Big( y  + \sum_{i=1}^d ( z^i S^i_t )^- \Big)^+
 - r^b_t \Big( y + \sum_{i=1}^d ( z^i S^i_t )^- \Big)^-  .
\ede

The following result describes the prices and replicating strategies for the hedger with an initial endowment $x$.
An analogous result holds for his counterparty, that is, the holder of the contract $(-A,-C)$ (see Propositions 4.1 and 4.2 in \cite{NR-2014}, as well as Theorem  \ref{inequality proposition for both positive initial wealth} in the foregoing subsection).

\bt \label{hedger ex-dividend price}
Let either Assumption \ref{changed assumption for lending cumulative dividend price} or Assumption \ref{xx2} be satisfied with $Q_t=t$. Then for any real number $x \geq 0$ and any contract $(A,C)$ admissible under $\PT^l$, the hedger's ex-dividend price satisfies $P^{h}(x,A,C) =  B^l (Y^{h,l,x} - x)  - C$ where $(Y^{h,l,x}, Z^{h,l,x})$ is the unique solution to the BSDE
\begin{equation} \label{BSDE with positive x for hedger}
\left\{
\begin{array}
[c]{l}
dY^{h,l,x}_t = Z^{h,l,x,\ast}_t \, d \wt S^{l,{\textrm{cld}}}_t
+\wt{f}_l \big(t, Y^{h,l,x}_t, Z^{h,l,x}_t \big)\, dt + dA^{C,l}_t, \medskip\\
Y^{h,l,x}_T=x.
\end{array}
\right.
\end{equation}
The hedger's unique replicating strategy $\phi $ can be obtained from  $(Y^{h,l,x}, Z^{h,l,x})$. Specifically,
$\phi$  equals $\phi = \big(\xi^1,\dots ,\xi^d, \psi^{l}, \psi^{b},\psi^{1,b},\dots ,\psi^{d,b}, \eta^b, \eta^l\big)$
where, for every $t\in[0,T]$ and $i=1,2,\ldots,d,$
\bde
\xi^i_{t}= Z^{h,l,x,i}_{t},  \quad
\psi^{i,b}_t =  -(B^{i,b}_t)^{-1} (\xi^i_t S^i_t)^+, \quad
\eta^b_t =-  (B^{C,b}_t)^{-1}C_t^+, \quad
\eta^l_t =(B^{C,l}_t)^{-1} C_t^-,
\ede
and
\begin{align*}
&\psi^{l}_t = (B^l_t)^{-1} \Big( B^l_tY^{h,l,x}_{t} + \sum_{i=1}^d ( \xi^i_t S^i_t )^- \Big)^+, \\
&\psi^{b}_t = - (B^r_t)^{-1} \Big(B^l_tY^{h,l,x}_{t}+ \sum_{i=1}^d ( \xi^i_t S^i_t )^- \Big)^-.
\end{align*}
\et

\proof
It is clear that the function $\wt{f}_l$ can be represented $\wt{f}_l(t,y,z)= g_l( t,y, \mathbb{S}_{t}z)$ where the
function $g_l$ is uniformly Lipschitzian. Furthermore, $\wt{f}_l$ satisfies (\ref{special generatorc}) and thus, in view of condition (ii) in Assumption \ref{changed assumption for lending cumulative dividend price}, it satisfies the uniform $m$-Lipschitz condition. Finally, it is obvious that $\wt{f}_l(t,0,0)=0$ for all $t \in [0,T]$, so that trivially
$\wt{f}_l(\cdot ,0,0)\in \wHzero$. In view of the preceding discussion, we conclude that Theorem \ref{appendix wellposedness theorem for lending special BSDE} can be used to establish the existence and uniqueness of a solution to BSDE \eqref{BSDE with positive x for hedger}.

Consequently, using the solution $(Y^{h,l,x}, Z^{h,l,x})$, we can find the hedger's ex-dividend price $P^{h}(x,A,C)$ and the replicating strategy $\phi$ when the process $A^{C,l}$ represents the discounted cash flows of a collateralized financial contract $(A,C)$ in a market model with funding costs introduced in Section 2.3 of \cite{NR-2014}. For the detailed financial interpretation of each component of the replicating strategy $\phi $, the interested reader is referred to Sections 2.2--2.3 in \cite{NR-2014}.
\endproof

\subsection{The Range of Fair Unilateral Prices} \label{sect5.3}

We conclude this paper by showing that Theorem \ref{comparison theorem 2x} is suitable for studying the bounds for fair or profitable  prices (see Definitions 3.9 and 3.10 in \cite{NR-2014}) of the collateralized contract when the two parties have, possibly different, initial endowments $x_1$ and $x_2$. For the sake of concreteness, we postulate here that the hedger and the counterparty have both non-negative initial endowments, which are denoted as $x_1$ and $x_2$, respectively. For other possible situations, we refer to Propositions 5.2--5.4 in \cite{NR-2014}.

\bt \label{inequality proposition for both positive initial wealth}
Let either Assumption \ref{changed assumption for lending cumulative dividend price} or Assumption \ref{xx2} be satisfied with $Q_t=t$. If the initial endowments satisfy $x_{1}\ge0$ and $x_{2}\ge0$, then for any contract $(A,C)$ admissible under $\PT^l$ we have, for every $t\in[0,T]$,
\be \label{eqq1}
P^{c}_t (x_{2},-A,-C)\leq P^{h}_t (x_{1},A,C). 
\ee
Hence the range of fair bilateral prices $[P^{c}_t (x_{2},-A,-C),\, P^{h}_t (x_{1},A,C)]$ is non-empty almost surely.
\et

\proof We assume that $x_{1}\ge0$ and $x_{2}\ge0$. From Proposition \ref{hedger ex-dividend price}, we know that $P^{h} (x_{1},A,C) =  B^l (Y^{h,l,x_{1}} - x_{1})  - C$ where $(Y^{h,l,x_{1}}, Z^{h,l,x_{1}})$ is the  unique solution to the BSDE
\be
\left\{
\begin{array}
[c]{l}
dY^{h,l,x_{1}}_t = Z^{h,l,x_{1},\ast}_t \, d \wt S^{l,{\textrm{cld}}}_t
+\wt{f}_l \big(t, Y^{h,l,x_{1}}_t, Z^{h,l,x_{1}}_t \big)\, dt + dA^{C,l}_t, \medskip\\
Y^{h,l,x_{1}}_T=x_{1}. \nonumber
\end{array}
\right.
\ee
Using similar arguments, but applied to $(x_2,-A,-C)$, we show that the counterparty's price equals $P^{c}(x_{2},-A,-C) =- (B^l (Y^{c,l,x_{2}} - x_{2})+C )$ where $(Y^{c,l,x_{2}}, Z^{c,l,x_{2}})$ is the unique solution to the BSDE
\be
\left\{
\begin{array}
[c]{l}
dY^{c,l,x_{2}}_t = Z^{c,l,x_{2},\ast}_t \, d \wt S^{l,{\textrm{cld}}}_t
+\wt{f}_l \big(t, Y^{c,l,x_{2}}_t, Z^{c,l,x_{2}}_t \big)\, dt - dA^{C,l}_t, \medskip\\
Y^{c,l,x_{2}}_T=x_{2}.\nonumber
\end{array}
\right.
\ee
Therefore, to prove \eqref{eqq1}, it suffices to establish the following
inequality
\bde
-B^l_t (Y^{c,l,x_{2}}_t - x_{2}) - C_t \leq B^l_t (Y^{h,l,x_{1}}_t - x_{1})  - C_t ,
\ede
which is manifestly equivalent to $-Y^{c,l,x_{2}}_t +  x_{2} \leq Y^{h,l,x_{1}}_t - x_{1}$.
If we denote $\bar{Y}^{h,l,x_{1}}:=Y^{h,l,x_{1}} - x_{1}$ and $\bar{Z}^{h,l,x_{1}}=Z^{h,l,x_{1}}$, then the pair $(\bar{Y}^{h,l,x_{1}},\bar{Z}^{h,l,x_{1}})$ is the unique solution of the following BSDE
\be\label{transferred BSDE for hedger 1}
\left\{
\begin{array}
[c]{l}
d\bar{Y}^{h,l,x_{1}}_t = \bar{Z}^{h,l,x_{1},\ast}_t \, d \wt S^{l,{\textrm{cld}}}_t
+\wt{f}_l \big(t, \bar{Y}^{h,l,x_{1}}_t+x_{1}, \bar{Z}^{h,l,x_{1}}_t \big)\, dt + dA^{C,l}_t, \medskip\\
\bar{Y}^{h,l,x_{1}}_T=0.
\end{array}
\right.
\ee
Similarly, $(\bar{Y}^{c,l,x_{2}}, \bar{Z}^{c,l,x_{2}}):=\big(-Y^{c,l,x_{2}}+x_{2},\, \bar{Z}^{c,l,x_{2}}_t=-Z^{c,l,x_{2}}\big)$ is the unique solution of the  BSDE
\be\label{transferred BSDE for counterparty 1}
\left\{
\begin{array}
[c]{l}
d\bar{Y}^{c,l,x_{2}}_t = \bar{Z}^{c,l,x_{2},\ast}_t \, d \wt S^{l,{\textrm{cld}}}_t
-\wt{f}_l \big(t, -\bar{Y}^{c,l,x_{2}}_t+x_{2}, -\bar{Z}^{c,l,x_{2}}_t \big)\, dt + dA^{C,l}_t , \medskip\\
\bar{Y}^{c,l,x_{2}}_T=0.
\end{array}
\right.
\ee
Note that (\ref{transferred BSDE for hedger 1}) and (\ref{transferred BSDE for counterparty 1}) have
the same term $dA^{C,l}_t$ and the same terminal condition $\eta = 0$. Also, we already know from the preceding subsection that the
generator $\wt{f}_l$ satisfies the conditions of the comparison Theorem \ref{comparison theorem 2x}.
Hence if either (see Theorem \ref{comparison theorem 2x})
\be \label{inequality for the lending drivers 1}
-\wt{f}_l \big(t, \bar{Y}^{h,l,x_{1}}_t+x_{1}, \bar{Z}^{h,l,x_{1}}_t \big)
\ge\wt{f}_l \big(t, -\bar{Y}^{h,l,x_{1}}_t+x_{2}, -\bar{Z}^{h,l,x_{1}}_t \big), \quad \PT^l\otimes \Leb-\aaee
\ee
or (see Remark \ref{remark for comparison theorem 2})
\be \label{inequality for the lending drivers 2}
-\wt{f}_l \big(t, \bar{Y}^{c,l,x_{2}}_t+x_{1}, \bar{Z}^{c,l,x_{2}}_t \big)
\ge\wt{f}_l \big(t, -\bar{Y}^{c,l,x_{2}}_t+x_{2}, -\bar{Z}^{c,l,x_{2}}_t \big), \quad \PT^l\otimes \Leb-\aaee
\ee
then the inequality $\bar{Y}^{h,l,x_{1}}\ge\bar{Y}^{c,l,x_{2}}$ holds $\PT^l\otimes \Leb$-a.e. To establish both
\eqref{inequality for the lending drivers 1} and \eqref{inequality for the lending drivers 2}, it suffices to
show that
\be \label{inequality for the lending drivers 3}
-\wt{f}_l \big(t, y+x_{1}, z \big)
\ge \wt{f}_l \big(t, -y+x_{2}, -z \big)\  \text{ for all } (y,z)\in \rr \times \rr^{d}, \quad \PT^l\otimes \Leb-\aaee
\ee
To complete the proof of the theorem, it suffices to note that the elementary inequality \eqref{inequality for the lending drivers 3} holds, as shown in Lemma \ref{lemnew1} below.
\endproof

\bl \label{lemnew1}
Assume that $x_{1}\ge0$ and $x_{2}\ge0$. Then the mapping $\wt{f}_l : \Omega \times [0,T] \times \mathbb{R}\times\mathbb{R}^{d} \to \mathbb{R}$ given by equation \eqref{drift function lending} satisfies \eqref{inequality for the lending drivers 3}.
\el

\proof 
Let us denote $\zzti = (B^l_t)^{-1} z^i S^i_t$. Then
\be
\begin{array}
[c]{ll}
\delta &:=\wt{f}_l \big(t, y+x_{1}, z \big)+\wt{f}_l \big(t, -y+x_{2}, -z \big)\medskip\\
&=- r^l_t (y+x_{1})+ f_l(t,B^l_t (y+x_{1}), z)- r^l_t (-y+x_{z})+ f_l(t,B^l_t (-y+x_{2}),-z)\medskip\\
&=- r^l_t (x_{1}+x_{2})-\sum_{i=1}^d r^{i,b}_t|\zzti |+r^l_t (\delta_{1}^{+}+\delta_{2}^{+})-r^b_t (\delta_{1}^{-}+\delta_{2}^{-})  \nonumber
\end{array}
\ee
where we denote
\be
\delta_{1}:= y+ x_{1}+\sumik_{i=1}^d (\zzti)^{-}, \quad
\delta_{2}:=- y+ x_{2}+\sumik_{i=1}^d (-\zzti)^{-}.\nonumber
\ee
From $r^l\leq r^b$, we have
\bde
\begin{array}
[c]{ll}
\delta&=- r^l_t (x_{1}+x_{2})-\sum_{i=1}^d r^{i,b}_t|\zzti |+r^l_t (\delta_{1}^{+}+\delta_{2}^{+})-r^b_t (\delta_{1}^{-}+\delta_{2}^{-})\medskip\\
& \leq - r^l_t (x_{1}+x_{2})-\sum_{i=1}^d r^{i,b}_t|\zzti|+r^l_t (\delta_{1}+\delta_{2})\medskip\\
&=- r^l_t (x_{1}+x_{2})-\sum_{i=1}^d r^{i,b}_t|\zzti|+r^l_t (x_{1}+x_{2})+\sum_{i=1}^d r^l_t|\zzti|\medskip\\
&=\sum_{i=1}^d (r^l_t-r^{i,b}_t)|\zzti|\leq 0.\nonumber
\end{array}
\ede
We thus conclude that $\delta\leq0$, so that (\ref{inequality for the lending drivers 3}) is valid.
\endproof

\vskip 10 pt

\noindent {\bf Acknowledgement.}
The research of Tianyang Nie and Marek Rutkowski was supported under Australian Research Council's
Discovery Projects funding scheme (DP120100895).




\end{document}